\documentclass[preprint]{imsart}%ejs
\usepackage{amsthm,amsmath}
\RequirePackage{natbib}
\RequirePackage[colorlinks,citecolor=blue,urlcolor=blue]{hyperref}
\usepackage{amssymb}
\usepackage{accents,xcolor}
\usepackage{mathtools}
\usepackage{mleftright}% for brackets around arrays
\usepackage{lscape} % for inserting some pages as landscape
\usepackage{xfrac,enumerate}

\doi{10.1214/154957804100000000}
\pubyear{0000}
\volume{0}
\firstpage{0}
\lastpage{0}
\arxiv{0000.00000}

\startlocaldefs

\newcommand{\R}{\mathbb{R}}
\newcommand{\bigo}{\mathcal{O}}
\newcommand{\N}{\mathcal{N}}
\newcommand{\D}{\mathcal{D}}
\newcommand{\X}{\mathcal{X}}
\newcommand{\sgn}{\text{sgn}}
\newcommand{\rank}{\text{rank}}
\DeclarePairedDelimiterX{\norm}[1]{\lVert}{\rVert}{#1}
\DeclarePairedDelimiterX{\abs}[1]{\lvert}{\rvert}{#1}
\DeclarePairedDelimiterX{\0norm}[1]{\lVert}{\rVert_{0}}{#1}
\DeclarePairedDelimiterX{\1norm}[1]{\lVert}{\rVert_{1}}{#1}
\DeclarePairedDelimiterX{\2norm}[1]{\lVert}{\rVert_{2}}{#1}
\DeclarePairedDelimiterX{\nnorm}[1]{\lVert}{\rVert_{n}}{#1}
\DeclarePairedDelimiterX{\2nnorm}[1]{\lVert}{\rVert_{n}^2}{#1}

\newtheorem{definition}{Definition}[section]
\newtheorem{lemma}[definition]{Lemma}
\newtheorem{theorem}[definition]{Theorem}

\newtheorem{Coro}[definition]{Corollary}
\theoremstyle{definition} 

\newtheorem*{note}{Note}
\newtheorem*{remark}{Remark}

\endlocaldefs

\usepackage{Sweave}
\begin{document}
\Sconcordance{concordance:TVAS.tex:TVAS.Rnw:%
1 84 1 1 0 51 1}
\Sconcordance{concordance:TVAS.tex:./Section1.Rnw:ofs 137:%
1 119 1}
\Sconcordance{concordance:TVAS.tex:./Section2.Rnw:ofs 257:%
1 69 1}
\Sconcordance{concordance:TVAS.tex:./Section3.Rnw:ofs 327:%
1 181 1}
\Sconcordance{concordance:TVAS.tex:./Section4.Rnw:ofs 509:%
1 64 1}
\Sconcordance{concordance:TVAS.tex:./Section5.Rnw:ofs 574:%
1 140 1}
\Sconcordance{concordance:TVAS.tex:./Section6.Rnw:ofs 715:%
1 76 1}
\Sconcordance{concordance:TVAS.tex:TVAS.Rnw:ofs 792:%
143 10 1}
\Sconcordance{concordance:TVAS.tex:./AppendixA.Rnw:ofs 803:%
1 94 1}
\Sconcordance{concordance:TVAS.tex:./AppendixB.Rnw:ofs 898:%
1 68 1}
\Sconcordance{concordance:TVAS.tex:./AppendixC.Rnw:ofs 967:%
1 29 1}
\Sconcordance{concordance:TVAS.tex:./AppendixD.Rnw:ofs 997:%
1 451 1 1 13 1 2 76 1 1 8 1 2 147 1 1 8 1 2 64 1}
\Sconcordance{concordance:TVAS.tex:TVAS.Rnw:ofs 1742:%
158 12 1}

\begin{frontmatter}
% "Title of the paper"
\title{Synthesis and analysis in total variation regularization}
\runtitle{Synthesis and analysis in total variation regularization}

% indicate corresponding author with \corref{}
% \author{\fnms{John} \snm{Smith}\corref{}\ead[label=e1]{smith@foo.com}\thanksref{t1}}
% \thankstext{t1}{Thanks to somebody} 
% \address{line 1\\ line 2\\ printead{e1}}
% \affiliation{Some University}

\author{\fnms{Francesco} \snm{Ortelli}\corref{}\ead[label=e1]{francesco.ortelli@stat.math.ethz.ch}}\and \author{\fnms{Sara} \snm{van de Geer}\ead[label=e2]{geer@stat.math.ethz.ch}}
% \address{R\"{a}mistrasse 101\\ 8092 Z\"{u}rich\\ }
% \and
% \author{\fnms{Sara} \snm{van de Geer}\ead[label=e2]{geer@stat.math.ethz.ch}}
\address{R\"{a}mistrasse 101\\ 8092 Z\"{u}rich\\ \printead{e1}; \printead{e2}}

\affiliation{Seminar for Statistics, ETH Z\"{u}rich}

\runauthor{Ortelli, van de Geer}

\begin{abstract}
We generalize the bridge between analysis and synthesis estimators by \cite{elad07} to rank deficient cases. This is a starting point for the study of the connection between analysis and synthesis for total variation regularized estimators. In particular, the case of first order total variation regularized estimators over general graphs and their synthesis form are studied.

We give a definition of the discrete graph derivative operator based on the notion of line graph and provide examples of the synthesis form of $k^{\text{th}}$ order total variation regularized estimators over a range of graphs.
\end{abstract}

%\begin{keyword}[class=MSC]
%\kwd[Primary ]{}
%\kwd{}
%\kwd[; secondary ]{}
%\end{keyword}


\end{frontmatter}
\tableofcontents

\section{Introduction}\label{s1}

\subsection{Analysis and synthesis}
In the literature we encounter two main approaches to regularized empirical risk minimization, the analysis  and the synthesis approach.

Assume a model with Gaussian noise, i.e. $Y=f^0+\epsilon, \epsilon \sim \N_n(0, \sigma^2\text{I}_n), f^0\in \R^n$, and let $\D\in \R^{m \times n}$ denote a generic analyzing operator. For a vector $v \in \R^n$ we write $\norm{v}^2_n= \norm{v}^2_2/n$ and $\norm{v}_n=\norm{v}_2/\sqrt{n}$. The analysis estimator $\hat{f}_{\text{A}}$ of $f^0$ is defined as
$$ \hat{f}_{\text{A}}:= \arg \min_{f\in \R^n} \left\{\norm{Y-f}^2_n+2\lambda\norm{\D f}_1 \right\}, \lambda>0.$$

The rationale behind this kind of estimator is that we know or suspect that the true signal $f^0$ is s.t. $\1norm{\D f^0}$ is small. So, what the analysis estimator does is to make a tradeoff between the fidelity to the observed data in $\ell^2$-norm and the fidelity to the structure we think $f^0$ has (which is encoded in $\mathcal{D}$) in $\ell^1$-norm. It is called analysis estimator since the candidate estimator $f$ of the signal is analyzed, i.e. some aspects of it, the rows of $\D f$, are calculated and penalized during the estimation process.

On the other side there is the synthesis estimator $\hat{f}_{\text{S}}$ of $f^0$, which expresses another approach to estimation.
Let $\X\in \R^{n\times p}$ be a generic dictionary, whose columns constitute signal atoms.
The synthesis estimator $\hat{f}_{\text{S}}$ of $f^0$ is defined as

$$ \hat{f}_{\text{S}}=\X\arg\min_{\beta\in\R^p} \left\{ \2nnorm{Y-\X\beta}+2\lambda\1norm{\beta} \right\}.$$

The rationale behind this kind of estimator is that we know or suspect that the true signal $f^0$ can be written as a sparse linear combination of columns of $\X$, also called dictionary atoms, i.e. as $f^0=\X\beta^0$, where $\beta^0$ is s.t. $\0norm{\beta^0}$ (or $\1norm{\beta^0}$, see compressed sensing literature) is small. So what the synthesis estimator does, is to trade off the fidelity in $\ell^2$-norm of a linear combination of dictionary atoms to the observed data  and the $\ell^1$-norm of the coefficients of this linear combination. This estimator is called synthesis estimator since a candidate estimator $f=\X\beta$ is synthesized (i.e. constructed) from a linear combination of (possibly few) dictionary atoms.

In the literature, \cite{elad07} make a connection between these two approaches to estimation by proposing a way to obtain equivalent synthesis estimators from analysis estimators.

In this article we want to deepen the understanding of this connection and explore its implications for the total variation regularized estimators over general graphs, which are instances of analysis estimators. Some questions arise, for instance:
\begin{itemize}
\item Can we find a synthesis form for total variation regularized estimators?
\item If yes, how does this synthesis form look like?
\item How can the synthesis form of total variation regularized estimators favour and ease the visualization, the understanding and the proofs of theoretical properties for such estimators?
\end{itemize}

In this article we derive a framework to construct dictionaries giving place to synthesis estimators equivalent to the analysis form of  different examples of total variation regularized estimators, which, as far as we know, is a new contribution.
Moreover we explain how the synthesis approach to total variation regularization can favour the development of more accurate and interpretable theory for total variation regularized estimators.
Finally we expose some examples of dictionaries arising from total variation regularization problems.

The aim of this text is thus to convey, by considering a complementary approach, a different perspective on total variation regularization.

\subsection{Total variation regularized estimators}

Let $\vec{G}=(V,E), V=[n], E=\{e_1, \ldots, e_m\}$ be a directed graph, where $V$ is the set of its vertices and $E$ is the set of its edges.
The graphs we consider are directed, this means that any edge $e_i=(e_i^-, e_i^+), i \in [m]$ is directed from vertex $e_i^-$ to vertex $e_i^+$.

\begin{definition}[Incidence matrix of $\vec{G}$]\label{d01s1}
The incidence matrix $D\in \{-1,0,1\}^{m \times n}$ of the graph $\vec{G}$ is given by
$$ (D_{\vec{G}})_{ij}= \begin{cases} -1, & j = e_i^-\\
+1, & j= e_i^+ \\
0, & \text{else} \end{cases}, i \in [m], j \in [n].$$
\end{definition}

\begin{note}
Note that, for a signal $f\in \R^n$, $D_{\vec{G}}f$ computes the first order differences of $f$ across the directed edges of $\vec{G}$. Therefore it is also called first order discrete graph derivative operator and written $D_{\vec{G}}=:D^1_{\vec{G}}$.
\end{note}

We now want to generalize the concept of discrete graph derivative operator to a general order $k\in \mathbb{N}$, i.e. we want to find an expression for $D^k_{\vec{G}}$. To do so we need to introduce the notion of line graph.

\begin{definition}[Directed line graph, \cite{levi11}]\label{d02s1}
Let $\vec{G}=(V,E)$ be a directed graph.
 The directed line graph $\vec{L}(\vec{G})$ of $\vec{G}$ is defined as
$$ \vec{L}(\vec{G})=(E, E_2),$$
where
$$ E_2= \left\{(e_i, e_j)\in E \times E \middle| e_i^+=e_j^- \right\}.$$
\end{definition}

\begin{note}
The idea behind using the line graph is the following: when computing differences along the direction of the graph, we obtain a value for each edge of the graph.
Since we follow the direction of the edges, these values represent the first discrete derivative of the signal with respect to the graph, and constitute the vertices of the line graph.
Two vertices of a directed line graph are then connected by an edge if the corresponding two edges in the underlying directed graph point into the same direction, i.e. they neither diverge nor collide.
\end{note}

\begin{definition}[$k^{\text{th}}$ order discrete graph derivative operator] \label{d03s1}
Let $k \in \mathbb{N}$.
For a directed graph $\vec{G}=(V,E)$, we define the $k^{\text{th}}$ order discrete graph derivative operator as
$$ D^k_{\vec{G}} :=\begin{cases} D_{\vec{G}} \prod_{i=1}^{k-1} D_{{\vec{L}}^i(\vec{G})} , &k>1,\\
D_{\vec{G}}, & k=1,\end{cases}$$
where $ {\vec{L}}^i (\vec{G})= \vec{L}( {\vec{L}}^{i-1}(\vec{G})), i \ge 2$ and ${\vec{L}}^1(\vec{G})= \vec{L}(\vec{G})$ is the directed line graph of $\vec{G}$.
\end{definition}

\begin{definition}[$k^{\text{th}}$ order total variation]\label{d04s1}
Let $f\in \R^n$ be a vector and $\vec{G}$ a directed graph.
The quantity $\norm{D_{\vec{G}}^k f}_1$ is called $k^{\text{th}}$ order total variation of $f$ on the graph $\vec{G}$.
\end{definition}

\begin{note}
Note that we use $k$ in the meaning used by \cite{gunt17} but not in the meaning used by \cite{wang16}. Indeed, in the notation used by \cite{wang16}, $\norm{D^{(k+1)}_{\vec{G}}f}_1$ would be called the $(k+1)^{\text{th}}$ order total variation of $f$ on $\vec{G}$. Thus the first order total variation would be obtained with $k=0$.
\end{note}

\begin{definition}[$k^{\text{th}}$ order total variation regularized estimator over the graph $\vec{G}$]\label{d05s1}
Let $k\in \mathbb{N}$, and let $\vec{G}=(V,E)$ be a directed graph.
The $k^{\text{th}}$ order total variation regularized estimator over the graph $\vec{G}$ is defined as
$$ \hat{f}=\arg \min_{f \in \R^n} \left\{\norm{Y-f}^2_n+ 2n^{k-1} \lambda \norm{D^k_{\vec{G}}f}_1 \right\}, \lambda >0.$$
\end{definition}

\begin{remark}
Note that the $1^{\text{st}}$ order total variation regularized estimator does not depend on the orientation of the edges, since only the absolute value of the edge differences is penalized.
However, higher order total variation regularized estimators are dependent on the orientation of the edges of $\vec{G}$. Indeed, consider the case of the path graph with one branch.
There are three second order differences centered around the ramification point.
Depending on which of the three leaves of the graph is (chosen as) the root of the graph, then only two of them are going to be penalized.
\end{remark}

\subsection{Notation}

By $\text{I}_n$ we denote the $n \times n$ identity matrix.
By $ 1_n$ we denote a vector with $n$ entries, all of them being ones.
By $\D \in \R^{m \times n}$ and $\X\in \R^{n \times p}$ we denote a general analyzing operator, respectively a general dictionary.

Let $\D_i, i \in [m]$ denote the $i^\text{th}$ row of an analysing operator $\D \in \R^{m \times n}$ and let $U \subseteq [m]$ be a set of row indices of $\D$. Then we define $\D_U=\{\D_i\}_{i \in U}$ and $\D_{-U}=\{\D_i\}_{i \in -U}$, where $-U=[m]\setminus U$.

Let $\X_j, j \in [p]$ denote the $j^\text{th}$ column of an analysing operator $\X \in \R^{n \times p}$  and let $U \subseteq [p]$ be a set of column indices of $\X$. Then we define $\X_U=\{\X_j\}_{j \in U}$ and $\D_{-U}=\{\D_j\}_{j \in -U}$, where $-U=[p]\setminus U$.

Let $\mathcal{V} \subset \R^n$ be a linear space. By $\Pi_{\mathcal{V}}\in \R^{n \times n} $ we denote the orthogonal projection matrix onto $\mathcal{V}$ and by $A_{\mathcal{V}}:= \text{I}_n-\Pi_{\mathcal{V}}$ the respective antiprojection matrix.

In the case where $\mathcal{V}= \text{rowspan}(\D_U)$, we use the shorthand notations $\Pi_{U}:=\Pi_{\text{rowspan}(\D_U)}= \D_U' (\D_U \D_U')^{-1} \D_U$ and $A_U=A_{\text{rowspan}(\D_U)}$, where we assume that $\D_U \D_U'$ is invertible.

In the case where $\mathcal{V}= \text{colspan}(\X_U)$, we use the shorthand notations $\Pi_{U}:=\Pi_{\text{colspan}(\X_U)}= \X_U (\X_U' \X_U)^{-1} \X_U'$ and $A_U=A_{\text{colspan}(\X_U)}$, where we assume that $\X_U' \X_U$ is invertible.

For any matrix $M$, let $\N(M)$ denote its nullspace and $\N^{\perp}(M)$ the orthogonal complement of its nullspace.

\section{Tools}\label{s2}

To make the exposition and the development of later insights more fluid, we expose already here some simple tools.

\subsection{Lasso with some unpenalized coefficients}

For this subsection, we assume the linear model, i.e. that $f^0$ can be written as some linear combination of dictionary atoms. This means that we assume $f^0=\X\beta^0$, where $\X\in\R^{n\times p}$ and $\beta^0\in\R^p$. 

We are going to consider two problems:
\begin{itemize}

\item A Lasso problem with coefficients indexed by $U\subseteq[p]$ not being penalized, i.e.
$$ \hat{\beta}:= \arg\min_{\beta\in\R^p} \left\{\norm{Y-\begin{bmatrix}\X_U & \X_{-U}\end{bmatrix}\beta}^2_n+2\lambda\norm{\beta_{-U}}_1 \right\}, \lambda>0;$$

\item A variant of the above problem, where the part of the design matrix corresponding to the coefficients being penalized is multiplied by the antiprojection onto the linear span of the columns of the design matrix corresponding to the coefficients not being penalized, i.e.
$$ \hat{\beta}^{\Pi}:= \arg\min_{\beta\in\R^p} \left\{\norm{Y-\begin{bmatrix}\X_U & A_U \X_{-U}\end{bmatrix}\beta}^2_n+2\lambda\norm{\beta_{-U}}_1 \right\}, \lambda>0.$$

\end{itemize}
Write
$$\hat{f}:=\begin{bmatrix}\X_U & \X_{-U}\end{bmatrix}\hat{\beta}$$
and
$$\hat{f}^{\Pi}:=\begin{bmatrix}\X_U & A_U \X_{-U}\end{bmatrix}\hat{\beta}^{\Pi}.$$

\begin{lemma}[cf. also Exercies 6.9 in \cite{buhl11}]\label{l21}
We have that
$$\hat{f}=\hat{f}^{\Pi}.$$
\end{lemma}

\begin{proof}[Proof of Lemma \ref{l21}]
See Appendix \ref{appA}.
\end{proof}

\begin{remark}
Note that $\hat{\beta}_U$ depends on $\hat{\beta}_{-U}$ and thus on $\lambda$, while $\hat{\beta}_{U}^{\Pi}$ does not.
\end{remark}

\begin{note}
Lemma \ref{l21} means that, when we have a Lasso problem with some coefficients not being penalized, we can add arbitrary quantities in the linear span of the columns of the design matrix corresponding to the unpenalized coefficients to its columns corresponding to the penalized coefficients without changing the prediction properties of the estimator.
\end{note}

\subsection{A view on the Moore-Penrose pseudoinverse}

\begin{lemma}[A view on the Moore-Penrose pseudoinverse]\label{l22}
Let $\D\in\R^{n\times n}$ be an invertible matrix. We write $\X:=\D^{-1}$. Let $U\subseteq [n]$ be a subset of the row indices of $\D$ of cardinality $\abs{U}=u$. We write
$$ \begin{bmatrix} \X_U & \X_{-U} \end{bmatrix}= \begin{bmatrix} \D_U \\ \D_{-U} \end{bmatrix}^{-1},$$
where $\X_U\in \R^{n \times U}$ and $\X_{-U}\in \R^{n \times (n-u)}$.
Then the Moore-Penrose pseudoinverse $\D^{+}_{-U}\in\R^{n\times (n-u)}$ of $\D_{-U}\in\R^{(n-u) \times n}$ is given by
$$ \D^{+}_{-U}= A_U \X_{-U},$$

where $A_U=\text{I}_n - \X_U(\X_U'\X_U)^{-1}\X_U'$.
\end{lemma}

\begin{proof}[Proof of Lemma \ref{l22}]
See Appendix \ref{appA}.
\end{proof}

\begin{remark}
Lemma \ref{l22} tells us that the Moore-Penrose pseudoinverse of an underdetermined matrix of full row rank $\D_{-U} \in \R^{(n-u)\times n}$ can be obtained as follows:

\begin{enumerate}
\item Add $u$ linearly independent rows to $\D_{-U}$, i.e. add $\D_U \in \R^{u \times n}$ on top of $\D_{-U}$ to obtain the invertible matrix $\D$.
\item Invert $\D$ to obtain $\X$.
\item Do the antiprojection of $\X_{-U}\in \R^{n\times (n-u)}$ onto the column span of $\X_U \in \R^{n \times u}$.
\end{enumerate}
\end{remark}

\section{Analysis versus synthesis}\label{s3}

In this section we are going to discuss the relation between analysis and synthesis as exposed in \cite{elad07}. All this section is adapted from their article and will be the basis for an extension of their theory in Sections \ref{s4} and \ref{s5}.

The main point distinguishing the analysis from the synthesis approach is that the synthesis approach is \textbf{constructive}, i.e. it tells us how to construct the estimator by using a sparse linear combination of the dictionary atoms. This means that it allows us to perform model selection (and eventually refitting), i.e.  to select which columns of the dictionary are relevant for describing the signal.

Moreover, for the analysis method the dimension of the unknown $f\in\R^n$ is relatively small, while it can be very large for the unknown $ \beta \in \R^p$ in the synthesis method.

Note that in the analysis estimator all the rows of $\D f$ get the same weight when it comes to computing the penalty $\norm{\D f}_1$.

\begin{remark}
The article by \cite{elad07} considers penalties with general $\ell_p^p$-norm. In this paper only penalties with the $\ell^1$-norm are of interest.
\end{remark}

\subsection{The underdetermined case}
\begin{lemma}[Underdetermined case: $m<n$. Theorem 2 in \cite{elad07}]\label{l31}
Let $\D\in\R^{m\times n}, m<n$ be an analyzing operator of rank $m$. Let $\Pi_\D:= \Pi_{\text{rowspan}(\D)}$ denote the projection matrix onto the row space of $\D$ and let $A_\D:=\text{I}_n-\Pi_\D$ be the projection onto the kernel of $\D$. Then
$$
\hat{f}_A=\hat{f}_S+A_\D Y,
$$
where the synthesis estimator is obtained with the dictionary $\X=\D^+$ and $\D^+\in\R^{n\times m}$ is the Moore-Penrose pseudoinverse of $\D$.
\end{lemma}

We report the proof by \cite{elad07} of Lemma \ref{l31} because the intuition behind it is of relevance when finding an equivalent synthesis form for total variation regularized estimators.

\begin{proof}[Proof of Lemma \ref{l31}]
See Appendix \ref{appB}
\end{proof}

We would now like to derive a more handy expression for the synthesis estimator.

\begin{lemma}\label{l311}
Let $\D\in \R^{m \times n}$, where $m=n-u$, be s.t. $\rank(\D)=n-u$.
Let $A \in \R^u$ be a matrix of full rank, s.t. $\tilde{\D}:= \begin{bmatrix} A \\ \D \end{bmatrix}$ is invertible. Write $\X= \begin{bmatrix} \X_U & \X_{-U} \end{bmatrix}:= \tilde{D}^{-1}$. 
Then

\begin{eqnarray*}
\hat{f}_{\text{A}}&=& \arg \min_{f \in \R^n} \left\{ \norm{Y-f}^2_n + 2 \lambda \norm{\D f}_1 \right\}\\
&=& \begin{bmatrix} \X_U &  \X_{-U} \end{bmatrix} \arg \min_{\beta\in \R^n} \left\{\norm{Y-\begin{bmatrix}\X_U &  \X_{-U}\end{bmatrix}\beta}^2_n + 2\lambda \norm{\beta_{-U}}_1 \right\}.
\end{eqnarray*}

\end{lemma}

\begin{proof}[Proof of Lemma \ref{l311}]
See Appendix \ref{appB}
\end{proof}

\subsection{The square invertible case}

\begin{lemma}[Square invertible case: $m=n$ and $\D$ is non singular. Theorem 1 in \cite{elad07}]\label{l32}
The analysis estimator $\hat{f}_{\text{A}}$ with invertible $\D\in\R^{n\times n}$ and the synthesis estimator $\hat{f}_{\text{S}}$ with $\X=\D^{-1}$ are equivalent.

\end{lemma}

\begin{proof}[Proof of Lemma \ref{l32}]
See Appendix \ref{appB}
\end{proof}

\subsection{The overdetermined case}\label{ss26}

We are now going to expose the theory by \cite{elad07} on the overdetermined case, i.e. the case where $\D\in\R^{m\times n}$ has $m>n$ and the matrix $\D$ is assumed to be of rank $n$.

The way suggested to find a synthesis estimator equivalent to the analysis estimator is based on a geometric interpretation of the two methods. All this section is adapted from \cite{elad07} and summarizes some aspects of that article which are relevant for our research.

\medskip

Consider the analysis and the synthesis estimator and assume $\D$ is of \textbf{full column rank}, i.e. $\text{rank}(\D)=n$ and $\X$ is of \textbf{full row rank}, i.e. $\text{rank}(\X)=n$. Then the analysis and synthesis problems can be rewritten as:
\begin{itemize}
\item Analysis:\\
$$ \hat{f}_{\text{A}}(a)=\arg\min_{f: \nnorm{Y-f}\le a} \1norm{\D f};$$

\item Synthesis:\\
$$ \hat{f}_{\text{S}}(a)=\arg\min_{\beta: \nnorm{Y-\X\beta}\le a} \1norm{\beta}.$$

\end{itemize}

Since $\X$ is of full rank, both the analysis and the synthesis estimator are constrained to be in the same set of radius $a$, i.e. in
$$ \left\{f: \nnorm{Y-f}\le a \right\}=\left\{f: \exists \beta : f=\X\beta \wedge \nnorm{Y-\X\beta}\le a \right\}.$$

In the following we assume that $0\notin \left\{f: \nnorm{Y-f}\le a \right\}$, otherwise the origin is the obvious solution to both problems.

The level sets of the target functions of the alternative formulation of the analysis and synthesis problems are collections of concentric, centro-symmetric polytopes around the origin, which can be inflated or deflated by simple scaling.

\begin{itemize}
\item Analysis:\\
$$ \left\{f: \1norm{\D f}\le c \right\}= c \left\{f: \1norm{\D f}\le 1 \right\}=:c \Psi_\D ;$$

\item Synthesis:\\
\begin{eqnarray*} \left\{f: \exists \beta: f=\X\beta \wedge \1norm{\beta}\le c \right\}&=:& \X \left\{\beta: \1norm{\beta}\le c \right\}\\
&=&  c \X \left\{\beta: \1norm{\beta}\le 1 \right\}=: c \Phi_\X .
\end{eqnarray*}
\end{itemize}

The behavior of the analysis and synthesis estimators is determined by the geometry of these two polytopes. Indeed the estimators, when starting with polytopes scaled with a small enough $c$, can be found by inflating these polytopes by increasing $c$ until the inflated polytopes hit the boundary of the region of radius $a$ aournd $Y$.
Specifying the \textbf{canonical polytopes} $\Psi_\D, \Phi_\X$ is equivalent to specifying $\D$ and $\X$.
For this reason $\Psi_\D$ is called (canonical) analysis defining polytope and $\Phi_\X$ is called (canonical) synthesis defining polytope.

\medskip

We now introduce some polytope terminology.% for an $n$-dimensional polytope.

For an $n$-dimensional polytope we use the following terms:
\begin{itemize}
\item Boundary: $(n-1)$-dimensional manifold;
\item Facet: $(n-1)$-dimensional surface, i.e. an $(n-1)$-dimensional face;
\item Boundary of facets: $(n-2)$-dimensional faces;
\item \ldots;
\item Ridge: 2-dimensional face;
\item Edge: 1-dimensional face;
\item Vertex: 0-dimensional face.
\end{itemize}

\medskip

Let us first consider the analysis defining polytope.
Consider the subdifferential $\nu(f)$ of $\1norm{\D f}$, which is also the normal to the surface of the polytope. By the chain rule of the subdifferential we have that
$$ \nu(f)=\D'\text{sgn}(\D f),$$
where
$$ \sgn(x)\begin{cases} =1, & x >0,\\
\in [-1,1], & x =0 , \\
= -1, & x<0.
\end{cases}$$

\underline{\textbf{Remarks on $\nu(f)$}}:\\
\begin{itemize}
%\item $\nu(f)$ is defined for each $f$ having at least one nonzero coordinate;
\item $\nu(f)$ has a discontinuity wherever a coordinate of $\D f$ vanishes. This discontinuity is arbitrarily filled in by the sign function. Thus $\nu(f)$ is piecewise smooth. In particular:
\begin{itemize}
\item at the facets, $\nu(f)$ is smooth;
\item at other faces (lower dimensional), $\nu(f)$ is discontinuous.
\end{itemize}
\item If $f$ is orthogonal to some row in $\D$ then $\nu(f)$ has discontinuities.
\end{itemize}

Let $\delta\Psi_\D$ denote the boundary of $\Psi_\D$.

\begin{lemma}[Claim 1 in \cite{elad07}]\label{l33}
Let $f\in\delta\Psi_\D$, where $\Psi_\D$ is the (canonical) $n$-dimensional analysis defining polytope. Let $k\in[n]$ be the rank of the rows to which f is orthogonal. Then $f$ is strictly wihtin a face of dimension $(n-k-1)$ of $\Psi_\D$.
\end{lemma}

Lemma \ref{l33} gives us a recipe to obtain the vertices of $\Psi_\D$:
\begin{enumerate}
\item Choose $(n-1)$ linearly independent rows in $\D$;
\item Determine a vector in their 1-dimensional nullspace $v\in\R^n$;
\item Normalize $v$ s.t. $\1norm{\D v}=1$. This normalization yields two antipodal vertices.
\item Select one of these two antipodal vertices.
\end{enumerate}

\begin{note}
The number of vertices of $\Psi_\D$ is equal to twice  the number of possibilities to choose $(n-1)$ linearly independent rows in $\D$.
\end{note}

\medskip

Next we consider the synthesis defining polytope. 

\begin{lemma}[Claim 2 in \cite{elad07}]\label{l24}
The (canonical) synthesis defining polytope is obtained as the convex hull of the columns of $\pm \X$.
\end{lemma}

\begin{Coro}[Corollary 1 in \cite{elad07}]\label{c25}
Let $\X_k$ be a columnn of $\X$ that can be obtained as a convex combination of $\pm \X_{-k}$. Then the synthesis problem with dictionary $\X$ and  the synthesis problem with dictionary $\X_{-k}$ are equivalent.
\end{Coro}

\medskip

From the considerations on the analysis and synthesis problems in the overdetermined case the following Theorem \ref{t26} follows.

\begin{theorem}[Theorem 4 in \cite{elad07}]\label{t26}
For each $\ell^1$-penalized problem with analysis operator $\D\in\R^{m\times n}, m>n$ of full rank $n$ there exists a dictionary $\X=\X(\D)$ describing an equivalent synthesis problem. The reverse is not true.
\end{theorem}

\begin{remark}
The equivalent synthesis problem can be obtained by taking the vertices of the analysis defining polytope (one for each antipodal pair) and setting them as dictionary atoms.
\end{remark}

\section{How to go from the analysis to the synthesis form}\label{s4}
In this section we expose a recipe to derive a synthesis form of a general analysis estimator. This recipe is developed based on \cite{elad07} and is more general than their, since it can handle also analysis operators not being of full rank.

\subsection{Interpretation of the matrix inverse}\label{s41}

Let $\D\in \R^{n \times n}$ have rank $n$. We want to obtain the matrix inverse of $\D$.
For $i \in [n]$ we do:
\begin{enumerate}
\item Find $v_i \in \N(\D_{-\{i\}})\setminus \{0\} \subset \R^n$. (Note that $d_jv_i=0, \forall j \not= i$).
\item Normalize $v_i$, s.t. $d_iv_i=1$, to obtain $v^*_i= \frac{v_i}{d_i v_i}$. (This corresponds to setting $v^*_i= \frac{\sgn(d_iv_i)v_i}{\norm{\D v_i}_1}$.)
\end{enumerate}

Let $V^*=(v^*_1, \ldots, v^*_n)$. Then $\D V^*= \text{I}_n$, i.e. $V^*$ is a right inverse of $\D$ and thus $V^*= \D^{-1}$.

We thus see that the method proposed by \cite{elad07} for the overdetermined case amounts to taking all the invertible submatrices of $\D \in \R^{m \times n}, m >n, \rank(\D)=n$ and invert them.  Let us call $\tilde{\X}\in \R^{n \times \tilde{p}}$ the matrix collecting all the inverses. We normalize the columns of $\tilde{\X}$, s.t. $\norm{\D\tilde{x}_i}_1=1$. Then we prune $\tilde{\X}$ according to \cite{elad07}, i.e. we remove a column $\tilde{\X}_k$ of $\tilde{\X}$ if it is in the convex hull of $\{\pm \tilde{\X}_j\}_{j \not=k}$.
We denote by $\X \in \R^{p \times n}$ the pruned dictionary.

\subsection{General recipe to pass from the analysis to the synthesis form}

\begin{theorem}\label{t01s4}
Let $\D \in \R^{m \times n}$ be a matrix with $ \rank(\D)=r \le \min \{m,n\}$. 

Let $T_i$ denote a set of row indices of $\D$, s.t. $\abs{T_i}=r$ and $\rank(\D_{T_i})=r$. 

We need to go through the following steps to find the dictionary for an equivalent synthesis problem, corresponding to the analysis problem with analysis operator $\D$.

\begin{enumerate}
\item Find all possible $T_i$'s.
\item Find $n-r$ rows spanning $\N(\D)= \N(\D_{T_i}), \forall i$. Write them in the matrix $A \in \R^{(n-r)\times n}$.
\item Invert $B_i= \begin{pmatrix} A \\ \D_{T_i} \end{pmatrix}$ to find $ B_i^{-1}= \begin{pmatrix} J & \tilde{\X}^i \end{pmatrix}, \tilde{\X}^i \in \R^{n \times r}$.
\item Write $ \tilde{\X}= \{\tilde{\X}^i \}_{i }$.
\item Normalize the columns of $ \tilde{\X}$ to obtain the normalized dictionary $\X$, s.t.  the columns of $\D\X$ have $\ell^1$-norm equal to 1.
\item Prune the dictionary $\X$ according to \cite{elad07}, i.e. discard a columns with index $k$ if $\X_k$ is in the convex hull of $\pm \X_{-k}$.
\item Obtain the dictionary $\X$.%, where the coefficients corresponding to the columns in $J$ are not penalized.
\end{enumerate}

\end{theorem}

\begin{proof}
See Appendix \ref{appC}
\end{proof}

\begin{Coro}\label{c01s4}
Let $\X$ and $J$ be obtained according to Theorem \ref{t01s4}. Then
\begin{eqnarray*}
\hat{f}_{\text{A}}&=& \arg \min_{f \in \R^n} \{\norm{Y-f}^2_n+2 \lambda \norm{\D f}_1 \}\\
&=& \begin{bmatrix} J & \X \end{bmatrix} \arg \min_{\beta \in \R^p} \left\{ \norm{Y-\begin{bmatrix} J & \X \end{bmatrix} \beta }^2_n + 2 \lambda \norm{\beta_{-[n-r]}}_1   \right\}.
\end{eqnarray*}
\end{Coro}

\begin{proof}[Proof of Corollary \ref{c01s4}]
Corollary \ref{c01s4} follows by reversing the calculations made at the beginning of the proof of Theorem \ref{t01s4}.
\end{proof}

\begin{remark}
An advantage of Theorem \ref{t01s4} with respect to the theory exposed by \cite{elad07} is that it is able to unify the treatment of the underdetermined, the square invertible and the overdetermined cases in one single result.
\end{remark}

\begin{note}
Lemma \ref{l21} tells us that in the penalized part of the dictionary, i.e. in  $\X$, we can actually add to the dictionary atoms arbitrary quantities in the column span of $J$ (i.e. in the nullspace $\N(\D)$) without changing the prediction properties of the estimator.
\end{note}

\section{Analysis versus synthesis in first order total variation regularization}\label{s5}

In this section which implications the insights exposed in Section \ref{s3} and further developed in Section \ref{s4} have for a class of analysis estimators of great interest in the current statistical literature: the first order total variation regularized estimators over graphs, see for instance
\cite{hutt16,wang16,dala17, gunt17,vand18, orte18}.

Let $\vec{G}=(V,E), \abs{V}=n, \abs{E}=m$ be a directed graph. Recall that by $D_{\vec{G}}\in \R^{m \times n}$ we denote its incidence matrix. In the following we omit the subscript and just write $D\in \R^{m \times n}$.

Until now we exposed some theory for a general analysing operator $\D$. From now on we identify $\D$ with the incidence matrix $D_{\vec{G}}$  of a directed  graph $\vec{G}$, i.e. we choose $\D=D_{\vec{G}}=D$. We want to investigate which consequences this choice of the analysing operator has for the resulting dictionary $\X$. We denote the resulting dictionaries with the capital letter $X$.

We focus onto connected graphs. The extension of the considerations to follow to graphs with more than one connected component is  straightforward, since the considerations can be applied to each connected component separately.

The first question arising is, when which of the three cases of an analyzing operator (underdetermined, invertible and overdetermined) is relevant.

\begin{note}
It is widely known that if $D \in \R^{m \times n}$ is the incidence matrix of a connected graph, then $\rank(D)=n-1, \forall m \ge n-1$.
Therefore, the case of an invertible analysing operator does never appear in the context of  first order total variation regularized estimators over graphs. Indeed, an incidence matrix only computes edge differences of a signal and  is thus oblivious of its mean level.
\end{note}
We are now going to consider two cases:
\begin{itemize}
\item Connected tree graphs;

\item Connected non-tree graphs.

\end{itemize}

\subsection{Tree graphs}

\begin{definition}[Tree graph]\label{d01s5}
A tree graph is a connected graphs having no cycles.
\end{definition}

\begin{remark}
Definition \ref{d01s5} is equivalent to saying that a tree graph is a  connected graph with $n$ vertices and $n-1$ edges. This implies that the incidence matrix $D$ of a tree graph is s.t. $D\in\R^{(n-1)\times n}$ has rank $n-1$ and $\text{ker}(D)=\text{span}(1_n)$.
\end{remark}

\begin{Coro}\label{c01s5}
We have that

$$ \hat{f}_{\text{A}}= X \arg \min_{\beta\in \R^n} \left\{\norm{Y-X\beta}^2_n + 2 \lambda \norm{\beta_{-1}}_1  \right\},$$
where $X= \begin{bmatrix} A \\ D \end{bmatrix}^{-1}$, with $ A=(1,0, \ldots, 0)\in \R^n$, is the path matrix rooted at vertex 1 of the tree graph considered.
\end{Coro}

\begin{proof}[Proof of Corollary \ref{c01s5}]
Note that $X_1=1_n$ and by Lemma \ref{l22} $D^+=A_1 X_{-1}$.
Thus, Corollary \ref{c01s5} follows by Theorem \ref{t01s4} combined with Lemms \ref{l21}.
\end{proof}

\begin{remark}
In the literature we find two ways in which $A\in \R^{1 \times n}$ is chosen to build $\tilde{D}$.
\cite{tibs11} propose to choose $A=1_n/n$, which is in the kernel of $D$. The resulting dictionary is then $X= \begin{bmatrix} 1_n & D^+ \end{bmatrix}$.

\cite{qian16} on the opposite side proposes to choose, as we do, $A=(1,0, \ldots, 0)$, s.t. $\begin{bmatrix} A \\ D \end{bmatrix}$ is the rooted incidence matrix at vertex 1 of the (path) graph and $X$ is the path matrix of the (path) graph with reference vertex 1. We prefer this last option, since it is of more intuitive interpretation.

Lemma \ref{l21} tells us that the two forms of the synthesis estimator resulting from these methods are equivalent in terms of prediction,
\end{remark}

\begin{remark}
It is interesting to look at the notion of sparsity for signals supported on tree graphs. When using an approach in the style of \cite{qian16}, the cardinality $s_0$ of the true active set $S_0$ tells us how many jumps there are in the signal. However the piecewise constant regions in $f^0$ are $s^0+1$, exactly as many as the number of coefficients required to express $f^0$ as a linear combination $f^0=X\beta^0$, where $\norm{\beta^0}_0=s_0+1$.
In the sequel indeed we are going to argue that $\norm{\beta^0}_0$ is an appropriate measure of the sparsity of the signal, rather than $\norm{Df^0}_0$.

\end{remark}

By the interpretation of the matrix inverse exposed in Section \ref{s4}, we see that finding the $i^{\text{th}}$ column of $X$ amounts to deleting the $i^{\text{th}}$ row from $\tilde{D}$ and then finding a vector $v_i^*$:
$$ \tilde{D}_{-i}v_i^*=0 \text{ and } \tilde{D}_{i}v_i^*=1.$$

For $i=1$, we see that $v^*_1=1_n$ satisfies these equations.

For $i \in [n]\setminus \{1\}$, the deletion of $\tilde{D}_i$ translates to the tree graph as the deletion of the edge $e_{i-1}\in E$. Since the deletion of an edge from a tree graph causes it to become disconnected, we end up with $D_{-\{i-1\}}$ being the incidence matrix of a graph with two connected components, both of them being tree graphs.

Let $V_1, V_2\subset V, V_1 \cap V_2= \emptyset, V_1 \cup V_2= V$ be the partition of the vertices of the tree graph arising from the deletion of the edge $(i-1)$. In particular, let $V_1$ contain the root of the graph, to which by convention we  assign the index $\{1\}$.

To find $v^*_i$ we first look for a vector s.t. $D_{-\{i-1\}}v^*_i=0$. Such a vector must have two piecewise constant components, i.e. it must be s.t. 

$$(v^*_i)_j=\begin{cases} a, & j \in V_1,\\
b, & j \in V_2.\end{cases}$$
The condition $Av^*_i=0, i \not=1$ gives $a=0$, while the condition $D_{i-1}v_i^*=1$ gives $b=1$.

Thus the dictionary $X_{-1}\in \R^{(n-1)\times n}$ contains all the ways to partition $V$ into two sets by cutting an edge of the tree graph $\vec{G}$. The vertices of the element of the partition containing the root will then get value zero and the vertices of the other element of the partition will get value one.

\begin{note}
For a tree graph, all these possible partitions can also be seen as all the possible ways to select $(n-2)$ (linearly independent) rows from $D$, i.e. all the ways to discard a row from $D$.
\end{note}

\subsection{Non tree graphs}
We are now going to consider the case of connected graphs not being trees, i.e. containing cycles. We want to find an equivalent synthesis formulation of the total variation regularized estimator on a general connected nontree graph $\vec{G}$.

\begin{note}
At first sight, this case could look like a purely overdetermined case. However, it shares features of both the underdetermined and the overdetermined case. Indeed, the incidence matrix $D$ of a graph with $n$ vertices containing some cycles has at least $n$ rows but is of rank $n-1$. Thus, the incidence matrices of this category of graphs have both the property of neglecting some information about the signal and the property of delivering some redundant information. This means that we have to combine the theory by \cite{elad07} for the overdetermined case  and the underdetermined case to find an equivalent analysis estimator, as done in Section \ref{s4} in Theorem \ref{t01s4}.
\end{note}

The intuition developed for tree graphs allows us now to handle the case of  connected nontree graphs more easily. Connected nontree graphs are graphs s.t. $m \ge n$ and thus their incidence matrix, which is of rank $n-1$, is not of full rank.

The first step in the recipe given by \ref{t01s4} is to find all sets $T_i \subset[m], \abs{T_i}=n-1$, s.t. $\rank(D_{T_i})=n-1$.
This means that we have to find all possible sets of edges in $D$ forming a connected graph with $n$ vertices. These sets of edges will thus define tree graphs. Let us introduce the notion of spanning tree.

\begin{definition}[Spanning tree]\label{d02s5}
Let $\vec{G}=(V,E)$ be a connected graph. The subgraph $\vec{T}=(V,E')$ is a spanning tree if $E'\subseteq E$ is the maximal set of edges containing no cycle.
\end{definition}

Therefore, thanks to Definition \ref{d02s5}, $T_i$ acquires the meaning of spanning tree. Thus, $\forall T_i$, we can apply the procedure developed in the previous subsection, arbitrarily assigning to vertex $\{1\}$ the role of the root of the spanning tree.

Note that it is possible that some of the dictionary atoms resulting from different spanning trees are the same. Thus, the step of pruning the dictionary might be necessary. In particular, duplicate of dictionary atoms will arise when cutting some edge of two spanning trees $T_i, T_j, i\not=j$ gives place to the same partition of vertices of a graph.
It follows that the dictionary atom will be generated by all the possible ways to select $(n-2)$ linearly independent rows from $D$, i.e. all the possible ways to partition the graph $\vec{G}$ into two connected components, both of them being tree graphs.

More formally, this corresponds to all the different possible ways to choose a set of edge indices $\bar{S}$, s.t. $\rank(D_{-\bar{S}})=n-2$ and $\abs{-\bar{S}}=m-n-2$. Let $E_{\bar{S}}:=\{e_i \in E, i \in \bar{S} \}$. Then $(V, E \setminus E_{\bar{S}})$ is a graph with two connected components, both of them being trees.

We have proven the following corollary.

\begin{Coro}\label{c02s5}
Let $\vec{G}=(V,E)$ be a graph (may be a tree graph as well as a nontree graph). Let $\{\bar{S}_j\}_{j \in [p]}$ be the set of all unique sets of edge indices $\bar{S}_j \subset [m]$, s.t. $\abs{\bar{S}_j}=m-n+2$, $\rank(D_{-\bar{S}_j})=n-2$ and  $(V, E\setminus E_{\bar{S}_j})=((V_1)_j, (E_1)_j)\cup ((V_2)_j, (E_2)_j), ( V_1)_j \cap (V_2)_j= \emptyset, (V_1)_j \cup (V_2)_j=V$ and $\{1\}\in (V_1)_j$.
Let $X_j=v_j/\norm{Dv_j}_1$, where 
$$ (v_j)_i= \begin{cases} 0, & i \in (V_1)_j,\\
1, & i \in (V_2)_j,\end{cases}$$
and $X= \{X_j\}_{j \in p}$.

Then
\begin{eqnarray*}
\hat{f}_{\text{A}}&=& \arg \min_{f\in \R^n} \left\{ \norm{Y-f}^2_n + 2 \lambda\norm{Df}_1 \right\}\\
&=& \begin{bmatrix} 1_n & X \end{bmatrix} \arg \min_{\beta\in \R^{p+1}} \left\{\norm{Y-\begin{bmatrix} 1_n & X \end{bmatrix}\beta}^2_n+ 2\lambda \norm{\beta_{-1}}_1 \right\}.
\end{eqnarray*}

\end{Coro}

\begin{remark}
Note that $X$ as it is is not uniquely defined, since it is not perpendicular to $\N(D)$ (i.e. its columns are not centered) and thus depends on the amound of deviation by its centered version. Let $X$ and $\Xi$ be two dictionaries differing only in the amount of deviation by their centered versions $A_{\N(D)}X$ and $A_{\N(D)}\Xi$. Then $A_{\N(D)}X=A_{\N(D)}\Xi$.
\end{remark}

\begin{remark}
Perhaps astonishingly, the dimensionalities of the analysis and of the synthesis problems in the overdetermined case may be very different. Indeed, the analysis problem is $n$ dimensional, while the synthesis problem in the worst case is $ \binom{m}{n-2}+1$ dimensional, where $m$ is the number of edges of the graph considered.

This is due to the fact that the incidence matrix defines a polytope having at least $ \binom{n}{n-2}\asymp n^2$ pairs of antipodal vertices in the overdetermined case.

The dictionary of the equivalent synthesis problem then contains an atom for each of these pairs of antipodal verices making its dimensionality much larger than $n$.
\end{remark}
\section{Why synthesis? A small review of the literature.}\label{s6}

To our knowledge, the only attempt to analyze some theoretical properties of the total variation regularized estimator over a wide range of graphs, such as the two-dimensional grid, the star graph and the $k^{\text{th}}$ power graph of the cycle, has been done by \cite{hutt16}.
The notion of sparsity used in their result coincides with the number of nonzero differences of the signal values across the $m$ edges of the graph $\vec{G}$ considered.
This number is also called cut metric and can be written as $\norm{D_{\vec{G}} f}_0$, where $D_{\vec{G}}$ is the incidence matrix of the graph $\vec{G}$ considered.

\cite{hutt16} obtain a general result for total variation regularized estimators over graphs, that they then narrow down according to different specific classes of graphs. Their result needs the definition of two quantities: the compatibility factor of $D_{\vec{G}}$ and the inverse scaling factor of $D_{\vec{G}}$.
Please note that now $\vec{G}$ has a general number of edges $m$, i.e. $D_{\vec{G}}\in\{-1,0,1\}^{m\times n}$, $m\ge n-1$.

Let $S'\subseteq [m]$ and $s'=\abs{S'}$.

\begin{definition}[Compatibility factor of $D_{\vec{G}}$, \cite{hutt16}]
The compatibility factor $\tilde{\tilde{\kappa}}_{S'}(D_{\vec{G}})$ of $D_{\vec{G}}$ for $S'$ is defined as
$$
\tilde{\tilde{\kappa}}_{S'}(D_{\vec{G}}):=\begin{cases}
\inf_{f\in\R^n}\frac{\sqrt{s'}\norm{f}_2}{\norm{(Df)_{S'}}_1} & , S'\not= 0,\\
1 &, S'=0.
\end{cases}
$$
\end{definition}

Let $D^+_{\vec{G}}\in\R^{n\times m}$ denote the Moore-Penrose pseudoinverse of $D_{\vec{G}}$ and let $d^+_1, \ldots, d^+_m$ denote its columns.

\begin{definition}[Inverse scaling factor of $D_{\vec{G}}$, \cite{hutt16}]
The inverse scaling factor $\rho(D_{\vec{G}})$ of $D_{\vec{G}}$ is defined as
$$ \rho(D_{\vec{G}}):=\max_{j\in[m]}\norm{d^+_j}_2.$$
\end{definition}

\cite{hutt16} prove that, if the tuning parameter $\lambda$ of the Edge Lasso is chosen as $\lambda\asymp \frac{\sigma}{n} \rho(D_{\vec{G}}) \sqrt{\log m}$, then with high probability
\begin{eqnarray*}
\norm{\hat{f}-f^0}^2_n&\le& \inf_{\substack{S'\subseteq [m]\\ f\in\R^n}}\left\{\norm{f^0-f}^2_n+ 4 \lambda \norm{(Df)_{-S'}}_1\right\} \\
& &+ \bigo\left(\frac{\sigma^2}{n} \frac{\rho^2(D_{\vec{G}})}{\tilde{\tilde{\kappa}}^2_{S'}(D_{\vec{G}})} s' \log m\right).
\end{eqnarray*}

The rates for the specific classes of graphs can be obtained by finding lower bounds on $\tilde{\tilde{\kappa}}_{S'}(D_{\vec{G}})$ and upper bounds on $\rho(D_{\vec{G}})$. The results obtained by \cite{hutt16} are summarized in Table \ref{tab1}.

\begin{table}\label{tab1}
\centering
\caption{Rates by \cite{hutt16}}
\begin{tabular}{|l|l|l|l|l|}
\hline
 & path & 2D grid & star & cycle \\ \hline
 $\rho(D_{\vec{G}})$ & $\sqrt{n}$ & $\bigo(\sqrt{\log n})$ & $\le 1$ & $\le \sqrt{n}$ \\ \hline
 $\tilde{\tilde{\kappa}}_{S'}(D_{\vec{G}})$ & $\Omega(1), s'\ge 2$ & $\Omega(1), s'\ge 6$ & $ \Omega(1/\sqrt{s'})$ & $ \Omega(1)$ \\ \hline
 $\rho^2(D_{\vec{G}})/\tilde{\tilde{\kappa}}^2_{S'}(D_{\vec{G}})$ & $\bigo(n)$ & $\bigo(\log n)$ & $\bigo(s')$ & $\bigo(n)$ \\ \hline
 $m$ & $n-1$ & $n-2\sqrt{n}$ & $n-1$ & $n$ \\ \hline
 rate by \cite{hutt16}& $\sigma^2 s' \log n$ & $\frac{\sigma^2 s' \log^2 n}{n}$ & $\frac{\sigma^2 {s'}^2 \log n}{n}$ & $\sigma^2 s' \log n$ \\ \hline
 \end{tabular}
\end{table}

We argue that the approach by \cite{hutt16}, even though it delivers an almost optimal rate for the case of the two dimensional grid, is not very idoneous for an extension to general graphs. There are 3 main reasons:

\begin{enumerate}
\item The ratio between the compatibility factor and the inverse scaling factor according to \cite{hutt16} is too strong to handle certain graphs, e.g. the path or the cycle graph. Indeed we can look at the ratio $\tilde{\tilde{\kappa}}^2_{S'}(D_{\vec{G}}) / \rho^2(D_{\vec{G}})$ as an analogous to the strong compatibility constant
$$
\tilde{\kappa}^2_{S'}(D_{\vec{G}}):= \inf_{f\in\R^n}\frac{s'\norm{f}^2_2}{n \norm{(Df)_{S'}}^2_1}\approx
\tilde{\tilde{\kappa}}^2_{S'}(D_{\vec{G}}) / \rho^2(D_{\vec{G}}).$$
This quantity is apparently too small to ensure convergence of the MSE for the cases of the path and the cycle graphs. Thus, it would make sense to try to find a lower bound for the weak compatibility constant. The weak compatibility constant was introduced and used by \cite{dala17} and later exploited by \cite{orte18} to prove an oracle inequality for the total variation regularized estimator on a class of tree graphs.
The weak compatibility constant for the case where $\vec{G}$ is the path graph is defined as
$$
\kappa^2_{S'}(D_{\vec{G}})= \inf_{f\in\R^n} \frac{s'\norm{f}^2_n}{(\norm{(D_{\vec{G}} f)_{S'}}_1-\norm{(D_{\vec{G}} f)_{-S'}}_1)^2},
$$
and is at least as large as the strong compatibility constant $\tilde{\kappa}_{S'}(D_{\vec{G}})$.

%\texttt{In the denominator, do I have to write $\norm{Df}_1$ or $\norm{\beta}_1$?}

\item The cut metric is not a good measure of the sparsity of the signal, as explained in \cite{padi17}. Indeed, consider the two dimensional grid. There might be signals having few constant pieces but very large cut metrics.
It seems that a good measure of the sparsity of the signal could be the number of piecewise constant regions in the graph, rather than the number of jumps.

\item For general graph structures containing cycles, i.e. not being trees, and in particular for the cycle graph and the two dimensional grid, the candidate set $S'\subseteq [m]$ of edges across which the signal has nonzero differences has to satisfy some conditions to make sense. For instance, for the cycle graph it does not make sense to have $s'=1$. In addition all the jumps have to sum up to zero.

\end{enumerate}

The last two points speak against the utilization of the approach by \cite{hutt16}. Therefore, the alternative synthesis approach seems to be  idoneous for ensuring that the concerns exposed in the last two points are avoided and that a more coherent notion of the sparsity of the signal on a graph is utilized when proving sparse oracle inequalities or other results involving the sparsity of the signal.

% \SweaveInput{Section7}
% \SweaveInput{Section8}
% \SweaveInput{Section9}

%\section{Simulations}

\appendix

\section{Proofs of Section \ref{s2}}\label{appA}

\begin{proof}[Proof of Lemma \ref{l21}]
We start by analyzing the estimator $\hat{\beta}$.
We can write the data fidelity term as follows:
\begin{eqnarray*}
 & & \norm{Y-\begin{bmatrix}\X_U & \X_{-U}\end{bmatrix}\beta}^2_n\\
& = & \norm{\Pi_U Y+ A_U Y -\X_U \beta_U-\Pi_U \X_{-U}\beta_{-U}-A_U \X_{-U} \beta_{-U}}^2_n\\
& = & \norm{\Pi_U (Y-\X_{-U}\beta_{-U}-\X_U\beta_U)}^2_n + \norm{A_U(Y-\X_{-U}\beta_{-U})}^2_n
\end{eqnarray*}

It follows that 
$$ \hat{\beta}_{U}= (\X_U'\X_U)^{-1}\X_U'(Y-\X_{-U}\hat{\beta}_{-U})$$

and

$$ \hat{\beta}_{-U}= \arg\min_{\beta_{-U}\in\R^{p-u}} \left\{\norm{A_U(Y-\X_{-U}\beta_{-U})}^2_n+2\lambda \norm{\beta_{-U}}_1 \right\}.$$

Next we consider the estimator $\hat{\beta}^{\Pi}$.
We can write the data fidelity term as follows:
\begin{eqnarray*}
& & \norm{Y-\begin{bmatrix}\X_U & A_U \X_{-U}\end{bmatrix}\beta}^2_n\\
& = & \norm{\Pi_U Y+ A_U Y -\X_U \beta_U-A_U \X_{-U} \beta_{-U}}^2_n\\
& = & \norm{\Pi_U (Y-\X_U\beta_U)}^2_n + \norm{A_U(Y-\X_{-U}\beta_{-U})}^2_n.
\end{eqnarray*}

It follows that

$$ \hat{\beta}^{\Pi}_U= (\X_U'\X_U)^{-1}\X_U'Y= \hat{\beta}_U+(\X_U'\X_U)^{-1}\X_U'\X_{-U}\hat{\beta}_{-U}$$

and

$$ \hat{\beta}^{\Pi}_{-U}= \arg\min_{\beta_{-U}\in\R^{p-u}} \left\{\norm{A_U(Y-\X_{-U}\beta_{-U})}^2_n+2\lambda \norm{\beta_{-U}}_1 \right\}.$$

We now compute $\hat{f}$ and $\hat{f}^{\Pi}$. For $\hat{f}$ we have that
\begin{eqnarray*}
\hat{f}&=& \X_U\hat{\beta}_U+\X_{-U}\hat{\beta}_{-U}\\
&=& \X_U \hat{\beta}^{\Pi}_U- \X_U (\X_U'\X_U)^{-1}\X_U'\X_{-U}\hat{\beta}_{-U} + \X_{-U}\hat{\beta}_{-U}\\
&=& \X_U \hat{\beta}^{\Pi}_U + A_U \X_{-U} \hat{\beta}_{-U}.
\end{eqnarray*}

and for $\hat{f}^{\Pi}$ we have that
$$ \hat{f}^{\Pi}=\X_U \hat{\beta}_U^{\Pi}+ A_U \X_{-U} \hat{\beta}_{-U}.$$

Hence the lemma follows.
\end{proof}

\begin{proof}[Proof of Lemma \ref{l22}]
We first note that
$$ \begin{bmatrix} \X_{U} & \X_{-U} \end{bmatrix} \begin{bmatrix} \D_{U} \\ \D_{-U} \end{bmatrix}= \X_{U}\D_{U} + \X_{-U}\D_{-U} = \text{I}_n , $$
which means that $\X_{-U}\D_{-U}= \text{I}_n -\X_{U}\D_{U}$.

Moreover,
$$ \begin{bmatrix}\D_{U} \\ \D_{-U} \end{bmatrix} \begin{bmatrix} \X_{U} & \X_{-U}\end{bmatrix}= \begin{bmatrix} \D_{U}\X_{U} & \D_{U}\X_{-U} \\ \D_{-U}\X_{U} & \D_{-U}\X_{-U} \end{bmatrix}= \begin{bmatrix} \text{I}_{n-m} & 0 \\ 0 & \text{I}_m  \end{bmatrix}.$$

It is known that the Moore-Penrose pseudoinverse of a matrix exists and is unique. Thus, if we can show that $\D_{-U}^+$ of the form exposed in the formulation of the lemma satisfies the four Moore-Penrose equations, we show that it is the Moore-Penrose pseudoinverse of $\D_{-U}$.

\begin{enumerate}

\item The first criterion to satisfy is: $\D_{-U} \D_{-U}^+ \D_{-U}= \D_{-U} $.\\
We check it by writing
\begin{eqnarray*}
\D_{-U} \D_{-U}^+ \D_{-U}&=& \D_{-U} (\text{I}_n-\X_{U} (\X_{U}'\X_{U})^{-1}\X_{U}')\X_{-U} \D_{-U}\\ &=& \D_{-U}\X_{-U}\D_{-U} = \D_{-U},
\end{eqnarray*}
which proves that $\D^{+}_{-U}$ satisfies it.

\item The second criterion is $ \D_{-U}^+ \D_{-U} \D_{-U}^+= \D_{-U}^+$. \\
We check it by writing
\begin{eqnarray*}
&& \D_{-U}^+ \D_{-U} \D_{-U}^+ = \\
&=& (\text{I}_n-\X_{U} (\X_{U}'\X_{U})^{-1}\X_{U}')\X_{-U} \D_{-U} (\text{I}_n-\X_{U} (\X_{U}'\X_{U})^{-1}\X_{U}')\X_{-U}\\
&=& (\text{I}_n-\X_{U} (\X_{U}'\X_{U})^{-1}\X_{U}')\X_{-U} = \D_{-U}^+,
\end{eqnarray*}
which proves that $\D^{+}_{-U}$ satisfies it.

\item  The third criterion is $( \D_{-U} \D_{-U}^+)'= \D_{-U} \D_{-U}^+$. \\
We show that $\D_{-U} \D_{-U}^+$ is symmetric, which implies the above equation. Indeed,
\begin{eqnarray*}
\D_{-U} \D_{-U}^+ &=& \D_{-U} (\text{I}_n-\X_{U} (\X_{U}'\X_{U})^{-1}\X_{U}')\X_{-U} \\
&=& \D_{-U}\X_{-U} = \text{I}_m.
\end{eqnarray*}

\item The fourth criterion is $(\D_{-U}^+ \D_{-U})'= \D_{-U}^+ \D_{-U}$. \\
We show that $\D_{-U}^+ \D_{-U}$ is symmetric, which implies the above equation. Indeed, 
\begin{eqnarray*}
\D_{-U}^+ \D_{-U}&=& (\text{I}_n-\X_{U} (\X_{U}'\X_{U})^{-1}\X_{U}')\X_{-U} \D_{-U} \\
&=& (\text{I}_n-\X_{U} (\X_{U}'\X_{U})^{-1}\X_{U}')(\text{I}_n-\X_{U}\D_{U})\\
&=& \text{I}_n - \X_{U} (\X_{U}'\X_{U})^{-1}\X_{U}' + \X_{U}\D_{U} - \X_{U}\D_{U}\\
&=& \text{I}_n - \X_{U} (\X_{U}'\X_{U})^{-1}\X_{U}',
\end{eqnarray*}
which is symmetric.
\end{enumerate}

\end{proof}
\section{Proofs of Section \ref{s3}}\label{appB}

\begin{proof}[Proof of Lemma \ref{l31}, \cite{elad07}]
By assumption we have that $\D$ is of full column rank, thus $\X=\D^+=\D'(\D\D')^{-1}$ and $\D \X=\text{I}_n$. Let $\Pi_\D$ denote the orthogonal projection matrix onto the row space of $\D$ and $A_\D$ the corresponding antiprojection matrix.

We note that the analysis estimator can be written as
\begin{eqnarray*}
\hat{f}_{\text{A}} &=& \arg \min_{f\in\R^n} \left\{ \norm{Y-f}^2_n + 2 \lambda \norm{\D f}_1 \right\} \\
&=& \arg \min_{\Pi_\D f, A_\D f\in \R^n} \left\{ \norm{\Pi_\D Y-\Pi_\D f}^2_n+ \norm{A_\D Y-A_\D f}^2_n +  2 \lambda \norm{\D \Pi_\D f}_1  \right\}.
\end{eqnarray*}

We thus see that
$$A_\D\hat{ f}_{\text{A}}=A_\D Y$$
and
$$\Pi_\D\hat{ f}_{\text{A}}=\arg\min_{\Pi_\D f\in\R^n} \left\{ \norm{\Pi_\D Y-\Pi_\D f}^2_n+ 2 \lambda \norm{\D\Pi_\D f}_1 \right\}.$$

Notice that  $\Pi_\D f$ is spanned by the columns of $\X$ and can be written as $\Pi_\D f=\X\beta$, for some $\beta$.

Thus

\begin{eqnarray*}
\Pi_\D \hat{f}_{\text{A}} &=& \X \arg \min_{\beta\in\R^{m}} \left\{ \norm{\Pi_\D Y - \X \beta}^2_n + 2 \lambda \norm{\D \X\beta}_1 \right\}\\
&=& \X \arg \min_{\beta\in\R^{m}} \left\{ \norm{\Pi_\D Y - \X \beta}^2_n + 2 \lambda \norm{\beta}_1 \right\}\\
&=& \X \arg \min_{\beta\in\R^{m}} \left\{ \norm{\Pi_\D Y - \X \beta}^2_n + \norm{A_\D Y}^2_n + 2 \lambda \norm{\beta}_1 \right\}\\
&=& \X \arg \min_{\beta\in\R^{m}} \left\{ \norm{Y - \X \beta}^2_n + 2 \lambda \norm{\beta}_1 \right\}= \hat{f}_{\text{S}} .
\end{eqnarray*}
We get that $\hat{f}_{\text{A}}=\hat{f}_{\text{S}}+A_\D Y$.
\end{proof}

\begin{proof}[Proof of Lemma \ref{l311}]
We start by noting that $A_{\D}= A_{\text{rowspan}(\D)}= \Pi_{\N(\D)}$.
By Lemma 2.2 with
$$ \tilde{\D}= \begin{bmatrix} A \\ \D \end{bmatrix} \text{ and } \X= \begin{bmatrix} \X_U & \X_{-U} \end{bmatrix},$$
we get that $\D^+= A_U \X_{-U}$.

Since $\tilde{\D}\X= \text{I}_n$, we have that $\D \X_U=0$, i.e. the columns of $\X_U$ are in the kernel of $\D$ independently of the choice of $A$.
At the same time we also have $A \X_{-U}=0$, i.e. the columns in $\X_{-U}$ have to be orthogonal to $A$.

Thus we see that $A_{\D}= \Pi_{\N(\D)}= \Pi_{\text{colspan}(\X_U)}= \X_U (X_U' X_U)^{-1} X_U'$.

We have that

$$ \Pi_{\N(D)}Y=:\Pi_U Y = \X_U\arg\min_{\beta_U \in \R^{u}}\norm{Y-\X_U\beta_U}^2_2=\X_U\arg\min_{\beta_U \in \R^{u}}\norm{\Pi_U(Y-\X_U\beta_U)}^2_2.$$

By Lemma \ref{l31} combined with the proof of Lemma \ref{l21} we see that we can write

$$ \hat{f}_{\text{A}}= \begin{bmatrix} \X_U & A_U \X_{-U} \end{bmatrix} \arg \min_{\beta\in \R^n} \left\{\norm{Y-\begin{bmatrix}\X_U & A_U \X_{-U}\end{bmatrix}\beta}^2_n + 2\lambda \norm{\beta_{-U}}_1 \right\},$$
where $\X_U$ is in the kernel of $\D$.

If we  choose $A$, s.t. $\text{span}(A)= \N(\D)$, then we have that $\X_{-U}$ and $\X_U$ are orthogonal and thus $\X_{-U}$ is the Moore-Penrose pseudoinverse of $\D$.% We write it anyhow as $\X_{-U}=A_U \X_{-U}$.

If we do not choose $A$ to be in the kernel of $\D$, then $ A_U \X_{-U}\not= \X_{-U}$, i.e. $\X_{-U}$ as it is is not the Moore-Penrose pseudoinverse anymore.

By Lemma \ref{l21} we know that we can leave out the antiprojection matrix without changing the prediction properties of the estimator.
We thus get that

$$ \hat{f}_{\text{A}}= \begin{bmatrix} \X_U &  \X_{-U} \end{bmatrix} \arg \min_{\beta\in \R^n} \left\{\norm{Y-\begin{bmatrix}\X_U &  \X_{-U}\end{bmatrix}\beta}^2_n + 2\lambda \norm{\beta_{-U}}_1 \right\}.$$
%where $\X=\tilde{\D}^{-1}$, $\tilde{\D}= \begin{bmatrix} A \\ \D \end{bmatrix} $ and $A$ is chosen s.t. $\tilde{\D}$ is invertible, thus not necessarily in the kernel of $\D$.
\end{proof}

\begin{proof}[Proof of Lemma \ref{l32}, \cite{elad07}]
In the definition of the synthesis estimator set $\beta=\D f$. Then $f=\D^{-1}\beta=\X\beta$. We thus have that
$$ \hat{f}_{\text{A}}=\X\arg\min_{\beta\in\R^n} \left\{\norm{Y-\X\beta}^2_n + 2 \lambda \norm{\beta}_1 \right\}= \hat{f}_{\text{S}}.$$
The same argument can be used to derive an equivalent analysis formulation from the synthesis problem.

\end{proof}
\section{Proofs of Section \ref{s4}}\label{appC}

\begin{proof}[Proof of Theorem \ref{t01s4}]
Note that
$$ A_{\N(\D)}= \Pi_{\N^{\perp}(\D)}= \Pi_{\text{rowspan}(\D)}.$$

We have that
\begin{eqnarray*}
&&\norm{Y-f}^2_n+ 2 \lambda\norm{\D f}_1\\
&=& \norm{\Pi_{\N(\D)}(Y-f)}^2_n+ \norm{A_{\N(\D)}(Y-f)}^2_n + 2 \lambda \norm{\D A_{\N(\D)}f}_1.
\end{eqnarray*}

It follows that
$$ \Pi_{\N(\D)}\hat{f}_{\text{A}}= \Pi_{\N(\D)} Y$$
and 
$$ A_{\N(\D)} \hat{f}_{\text{A}}= \arg \min_{f \in \N^{\perp}(\D)} \left\{ \norm{A_{\N(\D)}(Y-f) }^2_n+ 2\lambda \norm{\D f}_1 \right\}.$$

Note that $\text{dim}(\N^\perp(\D))=r$ and the initial analysis problem hides in itself a pure least square part to estimate $\Pi_{\N(\D)}f^0$ and a regularized empirical risk minimization step to estimate $A_{\N(\D)}f^0$, where the search for the minimizer happens in a $r$-dimensional linear subspace of $\R^n$.

Moreover, note also that $\{f\in \R^n: \norm{\D f}_1\le 1\}$ is an unbounded set if $r\not= n$. In the above empirical risk minimization step however we restrict to a $r$-dimensional hyperplane, i.e. to $\{f\in \R^n: f \in \text{rowspan}(\D) \}$. Indeed,
\begin{eqnarray*}
&&\{f\in \R^n: \norm{\D f}_1\le 1\}\cap \{f\in \R^n: f \in \text{rowspan}(\D) \}\\
&=&\{f\in \R^n: \norm{\D f}_1\le 1 \wedge f \in \text{rowspan}(\D) \}=: \tilde{\Psi}_{\D}
\end{eqnarray*}
is a $r$-dimensional polytope in $\{f\in \R^n: f \in \text{rowspan}(\D) \}$.

We now want to apply Lemma \ref{l33}. Lemma \ref{l33} adapted to $\tilde{\Psi}_{\D}$ tells us that, for $f \in \partial \tilde{\Psi}_{\D}$, if $k$ is the rank of the rows to which $f$ is orthogonal and if $f$ is orthogonal to $A$, then $f\in \R^n$ is strictly within a face of dimension $(n-r-k-1)$ of the $(n-r)$-dimensional polytope $\tilde{\Psi}_{\D}$.

Thus, by the interpretation of the matrix inverse in Subsection \ref{s41}, we see that the vertices of $\tilde{\Psi}_{\D}$ can be found by first finding all the sets $T_i \subseteq [m]: \abs{T_i}=r, \N(\D)=\N(\D_{T_i})$ and a matrix $A\in \R^{(n-r)\times n}$ spanning $\N(\D)$ and then inverting the matrices $B_i:= \begin{pmatrix} A \\ \D_{T_i} \end{pmatrix}, \forall i$ to find $\begin{pmatrix} J & \tilde{\X} \end{pmatrix}$. Note that $J \in \R^{n\times (n-r)}$ is the same $\forall i$ and spans $\N(\D)$. Write $\tilde{\X}= \{\tilde{\X}^i\}_i$. The columns of $\tilde{\X}^i$ have to be normalized s.t. the resulting normalized columns of $X^i$ belong to the boundary $\partial  \tilde{\Psi}_{\D}$ of $\tilde{\Psi}_{\D}$. After pruning we obtain the dictionary $\X$, which contains the vertices of $\tilde{\Psi}_{\D}$.
\end{proof}
\section{Some examples} \label{s7}
In this section we want to expose some examples of the classical analysis formulations of the first, second and third order total variation regularized estimators and of their corresponding analysis form for the path graph, the path graph with one branch and the cycle graph (which is not a tree graph). This allows us to see an application of the theory developed and in particular of Theorem \ref{t01s4}.  Moreover, these examples might be the starting point for the study and the development of some theory analogous to the one by \cite{orte18} for higher order total variation regularized estimators and for graphs that are not trees, for instance the cycle graph.

What the follwing examples show, is that the case $k=1$ is special for the cycle graph, in the sense that the dictionary atoms for $k=2$ and $k=3$ obtained for the corresponding synthesis form of the total variation regularized estimator over the cycle graph are quite different from the dictionary atoms that we obtain for the same kind of estimator on the path graph. Thus, for higher order total variation regularized estimators there seem to be less space for handling the case of the cycle graph by recycling the theory eventually developed for the path graph and for tree graphs in general.
Indeed, for $\vec{G}$ being a tree graph, $\rank(D^k_{\vec{G}})=n-k, k< n$, while for $\vec{G}$ being a cycle graph, $\rank(D^k_{\vec{G}})=n-1, \forall k$.

\subsection{Path graph}
In this subsection let $\vec{G}= (\{1, \ldots, n\}, \{(1,2), \ldots, (n-1,n) \})$, the path graph with $n$ vertices.

\subsubsection{k=1}

For $\vec{G}$, we have that

$$ D_{ij}= \begin{cases} -1, & j=i \\
+1, & j=i+1 \\
0, & \text{else}, \end{cases}, i \in [n-1],$$
where for notational convenience we omit the subscript $\vec{G}$ on the matrix $D$.

We choose $A= (1, 0 \ldots, 0) \in \R^{1 \times n}$ and we write $B = \begin{pmatrix} A \\ D \end{pmatrix}$.

The matrix $V=B^{-1}$ is given by
$$ V_{ij}= 1_{\{j\le i\}}, i,j\in [n].$$

We now make an example for $n=8$.
\begin{itemize}
\item Analysis

%We have that the analysis operator is given by
The $1^{\text{st}}$ order discrete graph derivative operator for $\vec{G}$ is
$$ D^1_{\vec{G}}= \begin{pmatrix*}[r]-1 & 1 & & & & & & \\
 & -1 & 1 & & & & & \\
 & & -1 & 1 & & & & \\
& & & -1 & 1 & & & \\
 & & & & -1 & 1 & & \\
 & & & & & -1 & 1 & \\
 & & & & & & -1 & 1\\
 \end{pmatrix*}\in \R^{7\times 8}.$$

\item Synthesis

We have that

$$ V= \begin{pmatrix} 1 &   &   &   &   &   &   &   \\
1 & 1 &   &   &   &   &   &   \\
1 & 1 & 1 &   &   &   &   &   \\
1 & 1 & 1 & 1 &   &   &   &   \\
1 & 1 & 1 & 1 & 1 &   &   &   \\
1 & 1 & 1 & 1 & 1 & 1 &   &   \\
1 & 1 & 1 & 1 & 1 & 1 & 1 &   \\
1 & 1 & 1 & 1 & 1 & 1 & 1 & 1 \\
\end{pmatrix}. $$

\end{itemize}

\subsubsection{k=2}

For $\vec{G}$, we have that

$$ D^2_{ij}= \begin{cases} +1, & j\in \{i, i+2\} \\
-2, & j=i+1 \\
0, & \text{else}, \end{cases}, i \in [n-2],$$
where for notational convenience we omit the subscript $\vec{G}$ on the matrix $D$.

We choose $$A = \begin{pmatrix} 1 & 0 & 0  \ldots & 0\\
-1 & 1 & 0 & \ldots & 0 \end{pmatrix}
\in \R^{2 \times n}$$ and we write $B = \begin{pmatrix} A \\ D^2 \end{pmatrix}$.

The matrix $V=B^{-1}$ is given by $V_1= 1_n'$ and
$$ V_{ij}= 1_{\{j\le i\}} (i-j+1), i\in [n], j \in [n] \setminus \{1\} .$$

We now make an example for $n=8$.
\begin{itemize}
\item Analysis

The $2^{\text{nd}}$ order discrete graph derivative operator for $\vec{G}$ is
$$ D^2_{\vec{G}}= \begin{pmatrix*}[r]1 & -2 & 1 & & & & & \\
 & 1 & -2 & 1 & & & & \\
 & & 1 & -2 & 1 & & & \\
& & & 1 & -2 & 1 & & \\
 & & & & 1 & -2 & 1 & \\
 & & & & & 1 & -2 & 1\\
 \end{pmatrix*}\in \R^{6\times 8}.$$

\item Synthesis

We have that

$$ V= \begin{pmatrix} 1 &   &   &   &   &   &   &   \\
1 & 1 &   &   &   &   &   &   \\
1 & 2 & 1 &   &   &   &   &   \\
1 & 3 & 2 & 1 &   &   &   &   \\
1 & 4 & 3 & 2 & 1 &   &   &   \\
1 & 5 & 4 & 3 & 2 & 1 &   &   \\
1 & 6 & 5 & 4 & 3 & 2 & 1 &   \\
1 & 7 & 6 & 5 & 4 & 3 & 2 & 1 \\
\end{pmatrix}. $$

\end{itemize}

\subsubsection{k=3}

For $\vec{G}$, we have that

$$ D^3_{ij}= \begin{cases} -1, & j=i \\
+3, & j=i+1 \\
-3, & j=i+2 \\
+1, & j=i+3 \\
0, & \text{else}, \end{cases}, i \in [n-3],$$
where for notational convenience we omit the subscript $\vec{G}$ on the matrix $D$.

We choose $$A = \begin{pmatrix} 1 & 0 & 0 & 0 & \ldots & 0\\
-1 & 1 & 0 & 0 & \ldots & 0 \\
1 & -2 & 1 & 0 & \ldots & 0 \\
\end{pmatrix}
\in \R^{3 \times n}$$ and we write $B = \begin{pmatrix} A \\ D^3 \end{pmatrix}$.

The matrix $V=B^{-1}$ is given by $V_1= 1_n'$, $ V_{i2}= i-1, i \in [n]$ and
$$ V_{ij}= 1_{\{j\le i\}} \frac{(i-j+1)(i-j+2)}{2}, i\in [n], j \in [n] \setminus \{1,2\} .$$

We now make an example for $n=8$.

\begin{itemize}
\item Analysis

The $3^{\text{rd}}$ order discrete graph derivative operator for $\vec{G}$ is
$$ D^3_{\vec{G}}= \begin{pmatrix*}[r]-1 & 3 & -3 & 1 & & & & \\
 & -1 & 3 & -3 & 1 & & & \\
 & & -1 & 3 & -3 & 1 & & \\
& & & -1 & 3 & -3 & 1 & \\
 & & & & -1 & 3 & -3 & 1\\
 \end{pmatrix*}\in \R^{5\times 8}.$$

\item Synthesis

We have that

$$ V= \begin{pmatrix} 1 &   &   &   &   &   &   &   \\
1 & 1 &   &   &   &   &   &   \\
1 & 2 & 1 &   &   &   &   &   \\
1 & 3 & 3 & 1 &   &   &   &   \\
1 & 4 & 6 & 3 & 1 &   &   &   \\
1 & 5 & 10  & 6 & 3 & 1 &   &   \\
1 & 6 & 15 & 10  & 6 & 3 & 1 &   \\
1 & 7 & 21 & 15 & 10  & 6 & 3 & 1 \\
\end{pmatrix}. $$

\end{itemize}

\subsection{Path graph with one branch}

Let now $\vec{G}= (\{ 1, \ldots, n\}, \{(1,2), \ldots, (n_1-1,n_1), (b, n_1+1), \ldots, (n-1,n)  \})$ be the path graph with $n=n_1+n_2$ vertices with the main branch consisting of $n_1$ vertices and the side branch consisting of $n_2$ vertices and being attached at vertx $1 < b < n_1$.

\subsubsection{k=1}

We have that

$$ D_{ij}= \begin{cases} -1, & j=i \\
+1, & j=i+1 \\
0, & \text{else}, \end{cases}, i \in [n-1]\setminus \{n_1\},$$
 and
 $$ D_{ij}= \begin{cases} -1, & j=b \\
+1, & j=n_1+1 \\
0, & \text{else}, \end{cases}, i \in [n-1]\setminus \{n_1\}, i = n_1.$$

We choose $A= (1, 0 \ldots, 0) \in \R^{1 \times n}$ and we write $B = \begin{pmatrix} A \\ D \end{pmatrix}$.

The matrix $V=B^{-1}$ is given by
$$ V_{ij}= 1_{\{j\le i\}}-1_{\{n_1+1\le i \le n\}\cap \{b+1 \le j \le n_1\}} , i,j\in [n].$$

We make an example with $n=8$, $b=4$ and $n_1=6$.

\begin{itemize}
\item Analysis

The $1^{\text{st}}$ order discrete graph derivative operator for $\vec{G}$ is
$$ D^1_{\vec{G}}= \begin{pmatrix*}[r]-1 & 1 & & & & & & \\
 & -1 & 1 & & & & & \\
 & & -1 & 1 & & & & \\
& & & -1 & 1 & & & \\
 & & & & -1 & 1 & & \\
 & & & -1 & & & 1 & \\
 & & & & & & -1 & 1\\
 \end{pmatrix*}\in \R^{7\times 8}.$$

\item Synthesis

We have that

$$ V= \begin{pmatrix} 1 &   &   &   &   &   &   &   \\
1 & 1 &   &   &   &   &   &   \\
1 & 1 & 1 &   &   &   &   &   \\
1 & 1 & 1 & 1 &   &   &   &   \\
1 & 1 & 1 & 1 & 1 &   &   &   \\
1 & 1 & 1 & 1 & 1 & 1 &   &   \\
1 & 1 & 1 & 1 &   &   & 1 &   \\
1 & 1 & 1 & 1 &   &   & 1 & 1 \\
\end{pmatrix}. $$

\end{itemize}

\subsubsection{k=2}

For $\vec{G}= (\{1, \ldots, n\}, \{(1,2), \ldots, (n-1,n) \})$, the path graph with $n$ vertices, we have that

$$ D^2_{ij}= \begin{cases} +1, & j\in \{i, i+2\} \\
-2, & j=i+1 \\
0, & \text{else}, \end{cases}, i \in [n-2]\setminus \{n_1, n_1+1\},$$
and

$$ D^2_{ij}= \begin{cases} +1, & j\in \{b-1, n_1+1\} \\
-2, & j=b \\
0, & \text{else}, \end{cases}, i =n_1,$$
and

$$ D^2_{ij}= \begin{cases} +1, & j\in \{b, n_1+2\} \\
-2, & j=n_1+1 \\
0, & \text{else}, \end{cases}, i= n_1+1,$$

We choose $$A = \begin{pmatrix} 1 & 0 & 0  \ldots & 0\\
-1 & 1 & 0 & \ldots & 0 \end{pmatrix}
\in \R^{2 \times n}$$ and we write $B = \begin{pmatrix} A \\ D^2 \end{pmatrix}$.

The matrix $V=B^{-1}$ is given by $V_1= 1_n'$ and
\begin{eqnarray*} V_{ij}&=& 1_{\{j\le i \le n_1\}} (i-j+1)\\
&+& (i-n_1+b-j +1) 1_{\{i \ge n_1+1, j \le b\}} \\
&+& (i-j+1) 1_{\{n_1+1 \le j \le i \le n\}}, i\in [n], j \in [n] \setminus \{1\} .
\end{eqnarray*}

We make an example with $n=8$, $b=4$ and $n_1=6$.

\begin{itemize}
\item Analysis

The $2^{\text{nd}}$ order discrete graph derivative operator for $\vec{G}$ is
$$ D^2_{\vec{G}}= \begin{pmatrix*}[r]1 & -2 & 1 & & & & & \\
 & 1 & -2 & 1 & & & & \\
 & & 1 & -2 & 1 & & & \\
& & & 1 & -2 & 1 & & \\
 & & 1 & -2 &  &  & 1 & \\
 & & &  1 & &  & -2 & 1\\
 \end{pmatrix*}\in \R^{6\times 8}.$$

\item Synthesis

We have that

$$ V= \begin{pmatrix} 1 &   &   &   &   &   &   &   \\
1 & 1 &   &   &   &   &   &   \\
1 & 2 & 1 &   &   &   &   &   \\
1 & 3 & 2 & 1 &   &   &   &   \\
1 & 4 & 3 & 2 & 1 &   &   &   \\
1 & 5 & 4 & 3 & 2 & 1 &   &   \\
1 & 4 & 3 & 2 &  &  & 1 &   \\
1 & 5 & 4 & 3 &  &  & 2 & 1 \\
\end{pmatrix}. $$

\end{itemize}

\subsubsection{k=3}

For $\vec{G}= (\{1, \ldots, n\}, \{(1,2), \ldots, (n-1,n) \})$, the path graph with $n$ vertices, we have that

$$ D^3_{ij}= \begin{cases} -1, & j=i \\
+3, & j=i+1 \\
-3, & j=i+2 \\
+1, & j=i+3 \\
0, & \text{else}, \end{cases}, i \in [n-3] \setminus \{n_1, n_1+1, n_1+2\},$$
and
$$ D^3_{ij}= \begin{cases} -1, & j=b-2 \\
+3, & j=b-1 \\
-3, & j=b \\
+1, & j=n_1+1 \\
0, & \text{else}, \end{cases}, i =n_1,$$
and
$$ D^3_{ij}= \begin{cases} -1, & j=b-1 \\
+3, & j=b \\
-3, & j=n_1+1 \\
+1, & j=n_1+2\\
0, & \text{else}, \end{cases}, i = n_1+1,$$
and
$$ D^3_{ij}= \begin{cases} b-1, & j=b \\
+3, & j=n_1+1 \\
-3, & j=n_1+2 \\
+1, & j=n_1+3 \\
0, & \text{else}, \end{cases}, i =n_1+2.$$

We choose $$A = \begin{pmatrix} 1 & 0 & 0 & 0 & \ldots & 0\\
-1 & 1 & 0 & 0 & \ldots & 0 \\
1 & -2 & 1 & 0 & \ldots & 0 \\
\end{pmatrix}
\in \R^{3 \times n}$$ and we write $B = \begin{pmatrix} A \\ D^3 \end{pmatrix}$.

The matrix $V=B^{-1}$ is given by $V_1= 1_n'$, $ V_{i2}= i-1 1_{\{ i \le n_1\}}+ (i-1-n_1+b) 1_{\{i >n_1\}}, i \in [n]$ and
\begin{eqnarray*}
V_{ij}&=& 1_{\{j\le i\le n_1\}} \frac{(i-j+1)(i-j+2)}{2}\\
&+&  1_{\{i \ge j> n_1\}} \frac{(i-j+1)(i-j+2)}{2}\\
&+& 1_{\{i >n_1, j \le b \}}  \frac{(i-j-n_1+b+1)(i-j-n_1+b+2)}{2} , i\in [n], j \in [n] \setminus \{1,2\} .\\
\end{eqnarray*}

We make an example with $n=8$, $b=4$ and $n_1=6$.
\begin{itemize}
\item Analysis

The $3^{\text{rd}}$ order discrete graph derivative operator for $\vec{G}$ is
$$ D^3_{\vec{G}}= \begin{pmatrix*}[r]-1 & 3 & -3 & 1 & & & & \\
 & -1 & 3 & -3 & 1 & & & \\
 & & -1 & 3 & -3 & 1 & & \\
& -1 & 3 & -3 &  & & 1 & \\
 & & -1 & 3 &  &  & -3 & 1\\
 \end{pmatrix*}\in \R^{5\times 8}.$$

\item Synthesis

We have that

$$ V= \begin{pmatrix} 1 &   &   &   &   &   &   &   \\
1 & 1 &   &   &   &   &   &   \\
1 & 2 & 1 &   &   &   &   &   \\
1 & 3 & 3 & 1 &   &   &   &   \\
1 & 4 & 6 & 3 & 1 &   &   &   \\
1 & 5 & 10  & 6 & 3 & 1 &   &   \\
1 & 4 & 6 & 3  &  &  & 1 &   \\
1 & 5 & 10  & 6 &   &  & 3 & 1 \\
\end{pmatrix}. $$
\end{itemize}

\subsection{Cycle graph}
In this subsection let $\vec{G}= (\{1, \ldots, n\}, \{(1,2), \ldots, (n-1,n), (n,1) \})$ be the path graph with $n$ vertices. Notice that $\vec{L}(\vec{G})= \vec{G}$.

We start the considerations about the cycle graph with the following lemma.

\begin{lemma}\label{lc01}
Let $D \in \R^{n \times n}$, s.t. $\N(D)= \N(D')$. Then $ \forall k \in \mathbb{N}$ we have that $\N(D^k)= \N(D)$ and thus $ \rank(D^k)= \rank(D)$.
\end{lemma}

\begin{proof}[Proof of Lemma \ref{lc01}]
We prove Lemma \ref{lc01} by induction.

\underline{\textbf{Anchor}:}\\

It is known that $ \N(D')$ is the orthogonal complement of the image of $D$.
We thus have that
\begin{eqnarray*}
\N(D^2)&=& \N(D) \cup (\N(D)\cap \text{Im}(D))\\
&=& \N(D) \cup (\N(D')\cap \text{Im}(D))\\
&=&\N(D) \cup \{ 0 \}\\
&=& \N(D).\\
\end{eqnarray*}

\underline{\textbf{Step}:}\\

Assume that $ \N(D^k)= \N(D)$. Then 
\begin{eqnarray*}
\N(D^{k+1})&=& \N(D^k) \cup (\N(D)\cap \text{Im}(D))\\
&=&\N(D) \cup \{ 0 \}\\
&=& \N(D).\\
\end{eqnarray*}

\end{proof}

Let us now define the matrix $T \in \R^{n \times n}$ as
$$ T= \begin{pmatrix}( \text{I}_n)_{n} & ( \text{I}_n)_{-n} \end{pmatrix} .$$

We note that $T$ is invertible, $T'=T^{-1}$ and that $(T^k)'= T^{-k},\forall k \in \mathbb{N}$.

Note that, for an invertible matrix $B\in \R^{n \times n}$ we have that

$$ (MT^k)^{-1}= (T^k)^{-1} M^{-1}= T^{-k} M^{-1}.$$

Thus note that if $B= \begin{pmatrix} A \\ D \end{pmatrix}$, with $A \in \R^{(n-r)\times n}: A T^k= A \forall k \in \mathbb{N}$, then
$$ (BT^{k})^{-1}= \begin{pmatrix} (B^{-1})_{1:(n-r)} \\
T^{-k} (B^{-1})_{(n-r+1): n} \end{pmatrix}.$$

Since for $D$ being the incidence matrix of the cycle graph with $n$ vertices we have that $\N(D)= \N(D')= \text{span}(1_n)$ and all $D_{T_i}$ are obtained by rotation of the incidence matrix of the path graph with $n$ vertices rooted at vertex $1$, we have that it is enough to calculate the inverse of a $B_i$ and then rotate its last $(n-1)$ columns by left multiplication with $T^k, k \in [n-1]$.

Thus we know that it is enough to invert the matrix $ B_i \begin{pmatrix} A \\ D_{T_i} \end{pmatrix}$ for one $i$ to find its inverse $(B_i)^{-1}= \begin{pmatrix} J & \tilde{X}^i \end{pmatrix}$. Then 
$$ \tilde{X}= \{ T^k \tilde{X}^i \}_{k\in [n]}.$$

Let us define $T_i= [n]\setminus \{i\}$.
 In the following we are going to select an $\{i\}$ and to call $B$ the matrix $\begin{pmatrix} A \\ D^k_{T_i} \end{pmatrix}$, just omitting the subscript $i$. Moreover we are going to write $V:= B^{-1}$. Note that, due to Lemma \ref{lc01}, $A$ will always be selected to be $1_n'$, i.e. a row vector in the linear span of $1_n'$. In all the cases we are going to have that $v_1= 1_n/n$.

\subsubsection{k=1}

We select $T_i =n$. Note that the $j^{\text{th}}$ columns of $V$, also denoted $V_j=:v_j, j \in \{2, \ldots, n\}$ is a vector in $ \N(B_{-j})$.  To calculate it we can proceed as follows.

Delete the $j^{\text{th}}$ row of $B$. We first find a vector $u_j$ in  $\N(B_{-\{1,j\}})$. We see that
$$ u_j= ( \underbrace{0, \ldots, 0}_{j-1}, \underbrace{1, \ldots, 1}_{n-j+1})'$$
is in $\N(B_{-\{1,j\}})$. Moreover we see that $B_jv_j=1$, so we do not have to normalize. It remains to subtract from $u_j$ its mean $\bar{u}_j$, to find $v_j= u_j - \bar{u}_j$ in $\N(B_{-j})$.
We note that $\bar{u}_j= \frac{n-j+1}{n}, j \in \{2, \ldots, n\}$ and thus
\begin{eqnarray*} v_j&=& \left(\underbrace{- \frac{n-j+1}{n}, \ldots,- \frac{n-j+1}{n} }_{j-1}, \underbrace{1- \frac{n-j+1}{n}, \ldots, 1- \frac{n-j+1}{n}}_{n-j+1} \right)'\\
&=& \left(\underbrace{-1 + \frac{j-1}{n}, \ldots, -1 + \frac{j-1}{n}}_{j-1}, \underbrace{\frac{j-1}{n}, \ldots, \frac{j-1}{n}}_{n-j+1} \right)'.\\
\end{eqnarray*}

We now want to see how many such vertices it is possible to find.

Note that 
$$ \{T^k v_2\}_{k=0, 1, \ldots, n-1}= \{T^k v_n\}_{k=0,1, \ldots, n-1}$$
and that
$$ \{T^k v_{j}\}_{k=0, 1, \ldots, n-1}= \{T^k v_{n-j+2}\}_{k=0,1, \ldots, n-1}, j \in \mathbb{N}, j \le (n+1)/2 .$$

If $n$ is even, then we have that for $j= \frac{n+2}{2}$
$$ \{T^k v_{(n+2)/2}\}_{k=0, 1, \ldots, n/2-1}= \{T^k v_{(n+2)/2}\}_{k=n/2, \ldots, n-1}.$$

Thus, if $n$ is even we have $p= (\frac{n}{2}-1)n+ \frac{n}{2}= \frac{n(n-1)}{2}= \binom{n}{2}$ synthesis dictionary atoms, whose coefficients are penalized, and $n-r=1$ dictionary atom, whose coefficient is not penalized.

If $n$ is odd we have that 
$$ \{T^k v_{j}\}_{k=0, 1, \ldots, n-1}= \{T^k v_{n-j+2}\}_{k=0,1, \ldots, n-1}, j \in \mathbb{N}, j \le (n+1)/2 .$$

Thus we have $p= \frac{n-1}{2} n= \binom{n}{2}$ dictionary atoms, whose coefficients are penalized and $n-r=1$ dictionary atom, whose coefficient is not penalized.

It follows that, in general, $p=\frac{n(n-1)}{2}$ and $n-r=1$.

We now make an example with $n=8$.
\begin{itemize}
\item Analysis\\

The $1^{\text{st}}$ order discrete graph derivative operator for $\vec{G}$ is
$$ D^1_{\vec{G}}= \begin{pmatrix*}[r]-1 & 1 & & & & & & \\
 & -1 & 1 & & & & & \\
 & & -1 & 1 & & & & \\
& & & -1 & 1 & & & \\
 & & & & -1 & 1 & & \\
 & & & & & -1 & 1 & \\
 & & & & & & -1 & 1\\
 1& & & & & & &-1\\
 \end{pmatrix*}\in \R^{8\times 8}.$$
 
\item Synthesis

The following Figure \ref{fc1} represents the plot of the last $(n-1)$ columns of the matrix $V$ obtained as explained above.

\begin{figure}[h]\label{fc1}
\includegraphics{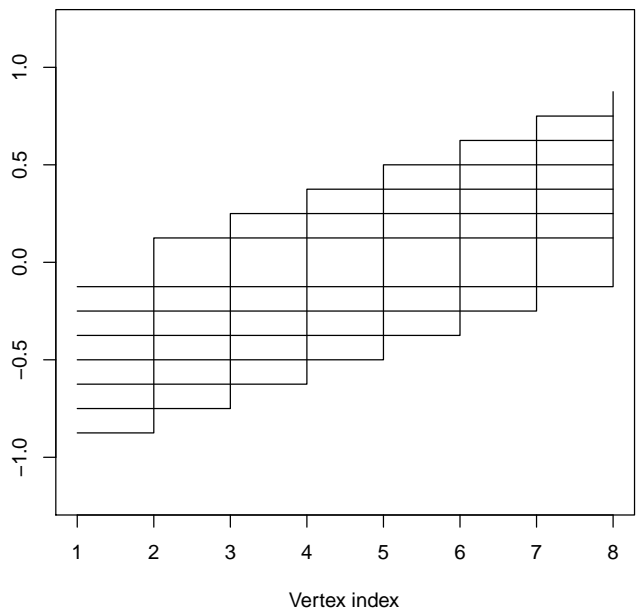}
\caption{Plot of the dictionary atoms whose coefficient are penalized, obtained by the inversion of $B$ for the case when $k=1$. These dictionary atoms correspond to the last $n-1$ columns of $V=B^{-1}$.}
\end{figure}
\end{itemize}

\subsubsection{k=2}
We select $T_i=T_n$.
We want to calculate $V= B^{-1}$. We see that $v_1= +_n'/n$.

To calculate $v_j, j \in \{2, \ldots, n\}$ we can proceed as follows.

We delete the $j^{\text{th}}$ row of $B$. We have to find $v_j$ in $\N(B_{-j})$, s.t. $B_jv_j=1$.

We first look for a vector $u_j$ in $\N(B_{-\{1,j\}})$. We see that
$$ u_j = (n-j+2, \{n-j+2-(i-1)\frac{n-j+1}{j-1}\}_{i=2}^{j-1},1, 2, \ldots, n-j+1)'$$
is such a vector. Note that it contains two segments with constant slope. We now want to normalize $u_j$, s.t. $B_j u^{*}_j= 1$. We have that
$$ B_j u_j= n-j+2-(j-2) \frac{n-j+1}{j-1}= \frac{n}{j-1}.$$

Thus $B_j u^*_j=1$, where 
$$ u^*_j=\frac{j-1}{n} u_j.$$

We now have to subtract from $u^*_j$ its mean $\bar{u}^*_j$, s.t. $u^*_j-\bar{u}^*_j\in \N(B_1)  $ as well. Since $\bar{u}^*_j=\frac{j-1}{n} \bar{u}_j $, we calculate $\bar{u}_j= \frac{\sum_{i=1}^n (u_j)_i}{n}$.

We have that
$$ \sum_{i=1}^n (u_j)_i = \sum_{i=1}^{n-j+2} i 1 \sum_{i=2}^{j-1}\left[(n-j+2) -(i-1) \frac{n-j+1}{j-1} \right].$$

We see that
$$ \sum_{i=1}^{n-j+2} i = \frac{(n-j+3)(n-j+2)}{2}$$
and that
\begin{eqnarray*}
\sum_{i=2}^{j-1} \left[ (n-j+2) -(i-1) \frac{n-j+1}{j-1} \right]&=& (j-2)(n-j+2)- \frac{n-j+1}{j-1}\sum_{i=1}^{j-2} i\\
&=& (j-2)(n-j+2)- \frac{(n-j+1)(j-2)}{2}\\
&=& (j-2) \left(n-j+2-\frac{n-j+1}{2} \right)\\
&=& \frac{(j-2)(n-j+3)}{2}.
\end{eqnarray*}

Thus,
$$ \sum_{i=1}^n= \frac{n(n-j+3)}{2}.$$

It follows that
$$ \bar{u}_j= \frac{n-j+3}{2}.$$

Thus we have that
$$ v_j= u_j^*- \bar{u}_j^*= \frac{j-1}{n} \left(u_j- \frac{n-j+3}{2} \right),$$
is the vector we are looking for.

Note that for $n$ even we have that
$$ \{T^k v_j\}_{\substack{j=2, \ldots, n/2\\ k = 0, 1, \dots, n-1}}= -\{T^k v_j\}_{\substack{j=n/2 +2, \ldots, n\\ k = 0, 1, \dots, n-1}},$$
and that
$$ \{T^k v_{n/2+1}\}_{ k = 0, 1, \dots, n/2-1}= -\{T^k v_{n/2+1}\}_{j=n/2, \ldots, n}.$$
Thus the part of the pruned dictionary whose coefficients are penalized contains $\binom{n}{2}$ atoms.

If $n$ is odd we have that
$$ \{T^k v_j\}_{\substack{j=2, \ldots, (n+1)/2\\ k = 0, 1, \dots, n-1}}= -\{T^k v_j\}_{\substack{j=(n+1)/2 +1, \ldots, n\\ k = 0, 1, \dots, n-1}}.$$

Thus the part of the pruned dictionary whose coefficients are penalized contains $\binom{n}{2}$ atoms.

\begin{itemize}
\item Analysis

The $2^{\text{nd}}$ order discrete graph derivative operator for $\vec{G}$ is
$$ D^2_{\vec{G}}= \begin{pmatrix*}[r]1 & -2 & 1 & & & & & \\
 & 1 & -2 & 1 & & & & \\
 & & 1 & -2 & 1 & & & \\
& & & 1 & -2 & 1 & & \\
 & & & & 1 & -2 & 1 & \\
 & & & & & 1 & -2 & 1\\
 1 & & & & &  & 1 & -2\\
-2 & 1 & & & &  &  & 1\\
 \end{pmatrix*}\in \R^{8\times 8}.$$

\item Synthesis

The following Figure \ref{fc2} represents the plot of the last $(n-1)$ columns of the matrix $V$ obtained as explained above.

\begin{figure}[h]\label{fc2}
\includegraphics{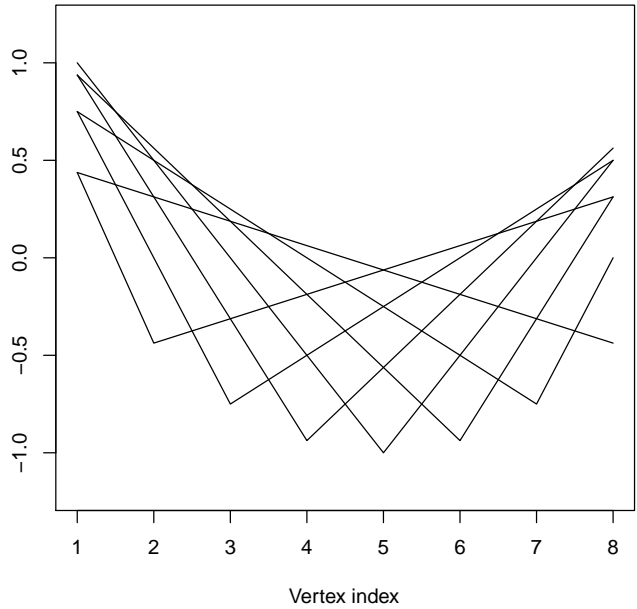}
\caption{Plot of the dictionary atoms whose coefficient are penalized, obtained by the inversion of $B$ for the case when $k=2$. These dictionary atoms correspond to the last $n-1$ columns of $V=B^{-1}$.}
\end{figure}

\end{itemize}

\subsubsection{k=3}

We now consider the set $T_{n-1}$. We consider the map $  j^0 \mapsto j,  \{2, \ldots, n\} \mapsto \{1, \ldots, n-1\}$ defined as
$$ j(  j^0)= \begin{cases}   j^0, &   j^0 \in \{2, \ldots, n-1 \}\\
1, &   j^0=n, \end{cases}$$
and let $  j^0(j)$ denote the inverse map. From now on we write for brevity $j$, when we mean $j(  j^0)$. We see that the vectors in $ \N(B_{-\{1,   j^0 \}})$ consist of a convex and a concave segment. The concave segment starts at vertex $1$ and ends at vertex $j(  j^0)$, while the convex segment starts at $j(  j^0)$ and ends at $1$.
We thus write
$$ u_j= \left(\{(i-1)(j-i)\}_{i=1}^j, -a\{(i-1)(n+1-i) \}_{i=j+1}^n \right),$$
where $a \in \R$ is a parameter needed to harmonize the two pieces.

We now want to find $a$. Since the convex and concave segment are symmetric, it is enough to check the condition only at one side. We have to have that $D_n u_j=0 \forall j \not= 1$.
We have that $$ D_n u_j= (u_j)_3-3 (u_j)_2 + 3 (u_j)_1-(u_j)_n.$$
Note that $(u_j)_1=0, \forall j$ and $(u_j)_n= -a (n-j), \forall j$, and thus the above expression becomes
$$ D_n u_j= (u_j)_3-3 (u_j)_2 + a(n-j).$$
We distinguish three cases:
\begin{itemize}
\item $j=1$;
\item $j=2$;
\item $j \ge 3$.
\end{itemize}

\begin{table}[h]
\begin{tabular}{|l|l|l|l|}
\hline
& $j=1$ & $j=2$ & $j\ge 3$ \\ \hline
   $i=3$ & -2a (n-2)  & -a(n-2) & 2(j-3)  \\ \hline
 $i=2$ & -a(n-1) & 0  & (j-2) \\ \hline
\end{tabular}
\label{tab71}
\caption{Values of $(u_j)_2$ and $(u_j)_3$ for the three cases $j=1$, $j=2$ and $j \ge 3$.}
\end{table}

In Table \ref{tab71} the values of $(u_j)_2$ and $(u_j)_3$ for the three cases are exposed. For these three cases we thus get the following.

\begin{itemize}
\item Case $j=1$. \\
$$ D_n u_1= -2a(n-2)+3a (n-1)+ a(n-1)= 2an.$$

\item Case $j=2$. \\
$$ D_n v_2=0.$$

\item Case $j \ge 3$. \\
$$ D_n v_j = 2 (j-3) - 3 (j-2) + a (n-j)= -j+ a(n-j).$$
Thus from the condition $D_n v_j=0$ we get that
$$ a= \frac{j}{n-j}.$$
\end{itemize}
Now we have to find a normalized version $u^*_j$ of $u_{j}$, s.t. $B_{  j^0} u^*_j=1$. We can see that we need to scale with a factor $\frac{n-j}{2n}$.

We thus have that
$$ u^*_j= \frac{n-j}{2n} \left( \{(i-1)(j-i)\}_{i=1}^j, - \frac{j}{n-j} \{ (i-j) ( n+1-i)  \}_{i=j+1}^n \right)$$
is in $ \N(B_{-  j^0(j)})$ and $B_{  j^0(j)}u_j=1$. It now remains to center the vector $u^*_j$, s.t. it is in $\N(B_1)$ as well.

We are going to need the formula
$$ \sum_{i=1}^k i^2= \frac{(k)(k+1)(2k+1)}{6}.$$

We thus calculate 
$$ \sum_{i=1}^n (u_j)_i= \underbrace{\sum_{i=1}^j(i-1)(j-i)}_{\text{I}}- \frac{j}{n-j}\underbrace{(i-j)(n+1-i)}_{\text{II}}.$$
We have that
\begin{eqnarray*}
\text{I}&=& \sum_{i=1}^j (i-1)(j-i)\\
&=& \sum_{i=1}^j \left[ ij-i^-j+i\right]\\
&=& -j^2 + (j+1)\sum_{i=1}^j i - \sum_{i=1}^j i^2 \\
&=& -j^2 + \frac{(j+1)^2j}{2}- \frac{(2j+1)j(j+1)}{6}\\
&=& \frac{(j-2)(j-1)j}{6}.
\end{eqnarray*}
Moreover,
\begin{eqnarray*}
\text{II}&=& \sum_{i=1}^j(i-j)(n+1-j)\\
&=& \sum_{i=1}^{n-j} i (n+1-j-i)\\
&=& (n+1-j) \sum_{i=1}^{n-j} i - \sum_{i=1}^{n-j} i^2\\
&=& \frac{(n+1-j)^2(n-j)}{2} - \frac{(n-j)(n-j+1)(2n-2j+1)}{6}\\
&=& \frac{(n+1-j)(n-j)(n-j+2)}{6}.
\end{eqnarray*}
It follows that
$$ \sum_{i=1}^n (u_j)_i= -\frac{j}{6} n (n-2j+3).$$
 
Therefore the vectors we are looking for are

$$ v_j= \left(\frac{n-j}{2n}\{(i-1)(j-i)\}_{i=1}^j, -\frac{j}{2n} \{(i-j)(n+1-i)\}_{i=j+1}^n \right) + \frac{j(n-j)(n-2j +3)}{12n}, j \in [n-1],$$
and
$$ (B^{-1})_{  j^0}= v_{j(  j^0)},   j^0\in \{2, \ldots, n\}.$$

We now want to find out how large $p$ is.

By making the changes of variables $j'=n-j$ and $i'=n-i+1$ we obtain that
\begin{itemize}
\item $ \frac{n-j}{2n} \{(i-1)(j-i) \}_{i=1}^j= \frac{j'}{2n} \{ (n-i')(i'-j'-1)\}_{i'=n}^{j'+1}$;

\item $ -\frac{j}{2n} \{ (n-i+1)(i-j)\}_{i=j+1}^{n}=-\frac{n-j'}{2n} \{i'(j'-i'+1) \}_{i'=j'}^1 $;

\item $ \frac{(n-j)j(n-2j+3)}{12n}= \frac{j'(n-j')(-n+2j'+3)}{12n}.$

\end{itemize}

Note that $(n-i')(i'-j'-1)= (i'-j')(n+1-j')-(n-j')$ and that $i'(j'-i'+1)= (i'-1)(j'-i') +j'$. We thus have that

\begin{itemize}
\item $ \frac{n-j}{2n} \{(i-1)(j-i) \}_{i=1}^j= \frac{j'}{2n} \{ (n-i'+1)(i'-j')\}_{i'=n}^{j'+1} - \frac{j'(n-j')}{2n}$;

\item $ -\frac{j}{2n} \{ (n-i+1)(i-j)\}_{i=j+1}^{n}=-\frac{n-j'}{2n} \{(i'-1)(j'-i') \}_{i'=j'}^1  - \frac{j'(n-j')}{2n}$.
\end{itemize}

Note that
$$ \frac{j'(n-j')(-n+2j'+3)}{12n}-\frac{j'(n-j')}{2n}= -\frac{j'(n-j')(n-2j'+3)}{12n}.$$

We thus see that
\begin{eqnarray*}
v_{j'}=v_{n-j}&=& -\left(-\frac{j'}{2n} \{ (n-i'+1)(i'-j')\}_{i'=n}^{j'+1}, \frac{n-j'}{2n} \{(i'-1)(j'-i') \}_{i=j'}^1  \right)\\
&-& \frac{j'(n-j')(n-2j'+3)}{12n}.
\end{eqnarray*}

By selecting $i^*= n-i'+j'+1, i'\in \{n, \ldots, j'+1\}$  and $i^*=j'-i'+1,  i'\in \{j', \ldots, 1\}$ (i.e. by exploiting the symmetry of the convex and the concave segments of $v_{j'}$), we have that

\begin{eqnarray*}
v_{j'}=v_{n-j}&=& -\left(-\frac{j'}{2n} \{ (n-i^*+1)(i^*-j')\}_{i^*=j+1}^{n}, \frac{n-j'}{2n} \{(i^*-1)(j'-i^*) \}_{i^*=1}^{j'}  \right)\\
&-& \frac{j'(n-j')(n-2j'+3)}{12n}\\
&=& -T^{n-j}v_j,
\end{eqnarray*}
 and thus we can apply the same considerations as we did for $k=2$. It follows that $p=\frac{n(n-1)}{2}$ and $n-r=1$ here as well.

\begin{itemize}
\item Analysis

The $3^{\text{rd}}$ order discrete graph derivative operator for $\vec{G}$ is
$$ D^3_{\vec{G}}= \begin{pmatrix*}[r]-1 & 3 & -3 & 1 & & & & \\
 & -1 & 3 & -3 & 1 & & & \\
 & & -1 & 3 & -3 & 1 & & \\
& & & -1 & 3 & -3 & 1 & \\
 & & & & -1 & 3 & -3 & 1\\
 1 & & & &  & -1 & 3 & -3\\
-3 & 1 & & &  &  & -1 & 3\\
3 & -3 & 1 & &  &  & & -1\\
 \end{pmatrix*}\in \R^{8\times 8}.$$
 
\item Synthesis

The following Figure \ref{fc3} represents the plot of the last $(n-1)$ columns of the matrix $V$ obtained as explained above.

\begin{figure}[h]\label{fc3}
\includegraphics{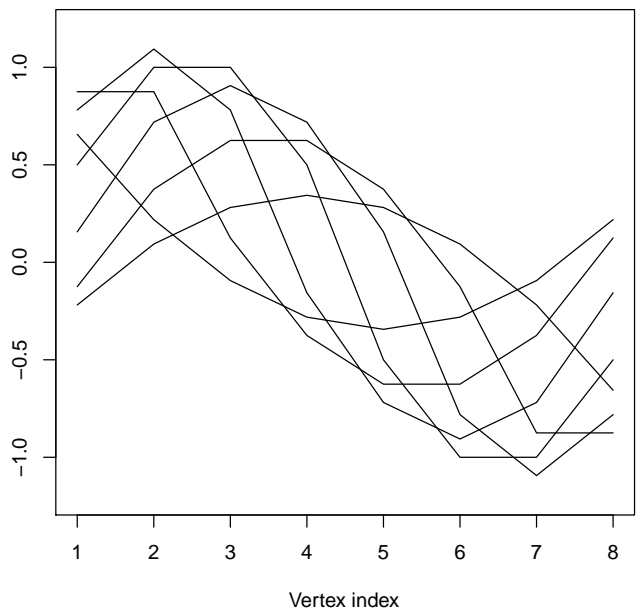}
\caption{Plot of the dictionary atoms whose coefficient are penalized, obtained by the inversion of $B$ for the case when $k=3$. These dictionary atoms correspond to the last $n-1$ columns of $V=B^{-1}$.}
\end{figure}

\end{itemize}

% \subsection{Lollipop graph, incoming}
% 
% \subsubsection{k=1}
% 
% \begin{itemize}
% \item Analysis
% 
% \item Synthesis
% 
% \end{itemize}
% 
% \subsubsection{k=2}
% 
% \begin{itemize}
% \item Analysis
% 
% \item Synthesis
% 
% \end{itemize}
% 
% \subsubsection{k=3}
% 
% \begin{itemize}
% \item Analysis
% 
% \item Synthesis
% 
% \end{itemize}
% 
% 
% \subsection{Lollipop graph, outgoing}
% 
% \subsubsection{k=1}
% 
% \begin{itemize}
% \item Analysis
% 
% \item Synthesis
% 
% \end{itemize}
% 
% \subsubsection{k=2}
% 
% \begin{itemize}
% \item Analysis
% 
% \item Synthesis
% 
% \end{itemize}
% 
% \subsubsection{k=3}
% 
% \begin{itemize}
% \item Analysis
% 
% \item Synthesis
% 
% \end{itemize}
% 
%\SweaveInput{AppendixE}

\newpage

\bibliographystyle{imsart-nameyear}

\bibliography{/Users/fortelli/PhD/library}

\end{document}